\documentclass[11pt,reqno]{amsart}
\usepackage{amssymb,mathrsfs,graphicx,subfigure, enumerate}
\usepackage{amsmath,amsfonts,amssymb,amscd,amsthm,bbm}
\usepackage{graphicx,colortbl,here}
\usepackage{extpfeil}
\usepackage{kotex}
\numberwithin{figure}{section}

\usepackage[utf8]{inputenc}
\usepackage[T1]{fontenc}

\def\d {{\partial}}

\def \rpsi_i {|\psi_i \rangle}
\def \lpsi_i {\langle \psi_i|}
\def \lrpsi_i{\langle \psi_i | \psi_i \rangle}
\def \rpsi_k {|\psi_k \rangle}
\def \lpsi_k {\langle \psi_k|}
\def \lrpsi_k{\langle \psi_k | \psi_k \rangle}

\DeclareMathOperator*{\esssup}{ess\,sup}
\DeclareMathOperator*{\Tr}{Tr}

\newcommand{\bbr}{\mathbb R}

\newcommand{\bbs}{\mathbb S}
\newcommand{\bbz}{\mathbb Z}
\newcommand{\bbe}{\mathbb E}

\newcommand{\bbp} {\mathbb P}

\newcommand{\p}{\partial}
\newcommand{\veps}{\varepsilon}

\newcommand{\ba}{\begin{aligned}}
\newcommand{\ea}{\end{aligned}}

\newcommand{\be}{\begin{equation}}
\newcommand{\ee}{\end{equation}}

\renewcommand{\d}{  {\textup{d}} }

\newcommand{\bbsd}{        {      \bbs^{d-1}  }    }
\newcommand{\bbn}{\mathbb{N}}

\newtheorem{theorem}{Theorem}[section]
\newtheorem{lemma}{Lemma}[section]

\newtheorem{proposition}{Proposition}[section]
\newtheorem{remark}{Remark}[section]

\newcommand{\lt}{\left}
\newcommand{\rt}{\right}
\newcommand{\bq}{\begin{equation}}
\newcommand{\eq}{\end{equation}}
\newcommand{\pa}{\partial}

\allowdisplaybreaks

\begin{document}

\title[Mean-field optimal control of the interacting particles on the sphere]{Optimal control of diffusive mean-field models for swarming particles on the sphere}

\author[J. Jung]{Jinwook Jung}
\address[J. Jung]{\newline Department of Mathematics and Research Institute for Natural Sciences \newline Hanyang University, 222 Wangshimni-Ro, 04763 Seoul, Republic of Korea}
\email{jinwookjung@hanyang.ac.kr}

\author[D. Kim]{Dohyun Kim}
\address[D. Kim]{\newline Department of Mathematics Education and Institute of Pure and Applied Mathematics, \newline Sungkyunkwan University, 25-2 Sungkyunkwan-Ro, 03063 Seoul, Republic of Korea}
\email{dohyunkim@skku.edu}

\thanks{\textbf{Acknowledgment.}
The work of D. Kim was supported by the National Research Foundation of Korea (NRF) grant funded by the Korean government (MSIT) (RS-2024-00454452).}

%\begin{abstract}
%We study a mean-field optimal control problem for a consensus (high-dimensional Kuramoto-type) dynamics on the unit sphere. At the macroscopic level, the agent density becomes a solution to a nonlinear Fokker--Planck equation on the unit sphere with a nonlocal alignment drift and diffusion. The control acts through a matrix-valued drift $\Psi(t,\omega)$ and a scalar gain $\kappa(t,\omega)$ multiplying the alignment operator. We consider a tracking-type cost functional that measures the $L^2$-distance between $f$ and a prescribed target density $z$, together with an $L^2$-penalization of the control effort.
%
%We first formulate the microscopic controlled interacting diffusion and its associated Liouville equation, and prove the existence of optimal controls for the finite-dimensional (Liouville) problem under Sobolev-type admissibility constraints. Next, for fixed control functions, we establish a quantitative stochastic mean-field limit: the first marginal of the $N$-particle Liouville solution converges to the mean-field solution in $L^\infty(0,T;L^1(\mathbb S^{d-1}))$ with rate $O(N^{-1/2})$, using a relative entropy method. Finally, we show that any weak accumulation point of microscopic optimal controls yields an optimal control for the mean-field problem, and the minimal costs converge as $N\to\infty$.
%\end{abstract}

\begin{abstract}
We study a mean-field optimal control problem for a consensus (high-dimensional Kuramoto-type) dynamics with diffusion on the unit sphere. The control acts through a prescribed drift field and an interaction gain, and the cost functional is given to track a given target density while penalizing the control effort. At the microscopic level, we formulate the corresponding controlled $N$-particle Liouville problem and establish the existence of optimal controls. For fixed controls, we obtain a quantitative stochastic mean-field limit showing that the one-particle marginal converges to the mean-field solution  with the convergence rate  $\mathcal O(1/\sqrt{N})$. Finally, we show that microscopic optimal controls approximate a mean-field optimal control: any weak limit of particle-level minimizers is optimal for the mean-field problem.
\end{abstract}


\makeatletter
\@namedef{subjclassname@2020}{%
  \textup{2020} Mathematics Subject Classification}
\makeatother


\date{\today}

\maketitle

%\tableofcontents

\section{Introduction}

%Collective alignment and synchronization phenomena arise in many multi-agent systems, ranging from swarming and flocking in biology to coordinated motion, attitude alignment, and consensus in engineered networks. When the individual states are constrained to lie on a manifold, geometric interaction laws become essential. A prototypical example is a consensus (high-dimensional Kuramoto-type) model on the unit sphere $\mathbb S^{d-1}$, in which each agent adjusts its state toward others through the tangential projection onto the sphere. Such models admit kinetic (mean-field) descriptions and display rich emergent behaviors, including alignment, phase transitions, and noise-induced effects.
%
%\vspace{2cm}

Collective behaviors of multi-agent systems such as swarming or opinion formation  have been extensively studied due to their relevance in biology, physics, and engineering, etc. \cite{L07, O06, V95}. When the number of interacting agents $N$ is sufficiently large, the dynamics of the group is commonly approximated by mean-field equations describing the time evolution of the density of agents  \cite{BCC12, HKLN19, S91}. For such systems or mean-field equations, the emergence of desired collective motions or patterns depends sensitively on the choice of initial data, interaction kernels and system parameters. This naturally motivates the incorporation of external interventions in the form of control functions; see \cite{CFPT13, CKPP19, G03}  and references therein for studies on construction of control functions which can steer the system toward prescribed macroscopic behaviors. 

Here, we take two aspects of the control problem into account. First, from an application perspective, external controls are typically subject to an optimality problem with respect to suitable cost functionals. Accordingly, one seeks controls that steer the agents toward desired collective states while minimizing a given cost functional.  Second, while the governing equations for dynamics of agents are mostly formulated on Euclidean space, the agents are subject to intrinsic geometric constraints, and their states evolve naturally on manifolds. 

Motivated by these considerations, we investigate the optimal control problem for multi-agent systems on some manifolds, both at microscopic and macroscopic (i.e., mean-field) scales. Although there is   substantial literature on emergent behavior of multi-agent systems on manifolds and their mean-field descriptions \cite{CLP15,  CH14, Lo09, ZZQ18, Zhu13},  optimal control problems for dynamical systems on manifolds have received comparatively little attention.

In this work, we address an optimal control problem for a consensus model on the unit sphere $\mathbb S^{d-1}$ originally introduced by  \cite{Lo09, O06}. 
We investigate the problem at both microscopic and macroscopic scales, and rigorously connect the two via a stochastic mean-field limit. Our starting point is a controlled mean-field PDE for the density $f(t,\omega)$:
\begin{equation}\label{A-0}
\partial_t f + \nabla_\omega \cdot \bigl(\Psi  \omega  f + \kappa   L[f]  f \bigr)
= \sigma \Delta_\omega f, \qquad \omega\in\mathbb S^{d-1},
\end{equation}
Here, $f=f(t,\omega)$ is the density in $\bbsd$ at time $t$, $\nabla_\omega$, $\nabla_\omega \cdot$ and $\Delta_\omega$ are the gradient, divergence and Laplace-Beltrami operator on $\bbsd$, respectively, the integration is the normalized surface measure $d\omega$ of $\bbsd$ so that $\int_{\bbsd} d\omega =1$ and $\sigma\geq0$ is a diffusion coefficient. Lastly, the nonlocal alignment operator is defined as 
\[
L[f](t,\omega) :=\int_{\mathbb S^{d-1}} \mathbb{P}(\omega)\omega_* f(t,\omega^*)d\omega_*,\quad  \mathbb{P}(\omega)\omega_* := \omega_* - \langle \omega,\omega_*\rangle \omega
\]
where $\mathbb{P}$ is the projection operator on the tangent space of $\bbsd$. For the two control functions, $\Psi=\Psi(t,\cdot)$ is a matrix-valued drift realized as a skew-symmetric matrix and $\kappa=\kappa(t,\cdot)$ is a scalar gain scaling the alignment interaction.

In what follows, we introduce our cost functional.  Given a target density $z(t,\omega)$, we consider a tracking-type cost functional of the form
\begin{align} \label{A-1}
\begin{aligned}
\mathcal J(f[\Psi,\kappa],\Psi,\kappa) &:=\frac{\alpha}{2}\int_0^T  \int_{\mathbb S^{d-1}}|f(t,\omega)-z(t,\omega)|^2 d\omega dt \\
&\quad+\frac{\beta}{2}\int_{\mathbb S^{d-1}}|f(T,\omega)-z(T,\omega)|^2 d\omega \\
&\quad+\frac12\int_0^T  \int_{\mathbb S^{d-1}}\bigl(\lambda_1|\Psi(t,\omega) |^2+\lambda_2|\kappa(t,\omega)|^2\bigr)d\omega dt,
\end{aligned}
\end{align}
where $\alpha,\beta,\lambda_1$, and $\lambda_2$ are non-negative constants and $f=f(t,\omega)$ solves \eqref{A-0} with given initial data $f_0$. Here, the cost functional $\mathcal J$ consists of three terms: (i) running tracking term, (ii) terminal tracking term and (iii) control regularization. Based on $L^2$-metric,  the first term penalizes the time-accumulated discrepancy between the evolving state density $f(t,\cdot)$ and the prescribed target density $z(t,\cdot)$ over the entire control horizon $[0,T]$. For the second term, it imposes a penalty on the final-time error between $f(T,\cdot)$ and $z(T,\cdot)$. Thus, it enforces endpoint accuracy ensuring that the controlled density reaches (or remains close to) the desired terminal configuration at time $t=T$. Lastly, the third term penalizes the magnitude (or energy) of the control inputs over the space and time. Each term regularizes the matrix-valued control field $\Psi$ and the scalar control $\kappa$, respectively. Such $L^2$-penalization prevents excessively large controls and promotes smooth and physically plausible actuation. Overall, we would say that $\mathcal J$ balances (i) trajectory-level tracking of the target density, (ii) terminal accuracy at the final time and (iii) regularized control expenditure.

 For this cost functional, our goal of the optimal control problem is to seek a minimizer $(\bar \Psi,\bar \kappa)$ over an admissible class of controls:
\[
\mathcal J (f[\bar \Psi,\kappa],\bar \Psi,\bar \kappa) = \min_{\Psi,\kappa\in \mathcal U} \mathcal J (f[\Psi,\kappa], \Psi,  \kappa).
\] 
Here, $\mathcal U$ is the set of admissible controls defined as
\begin{equation}  \label{A-1-2}
\mathcal U := \{ u \in L^\infty(0,T; W^{1,q}(\bbsd)) : \|u\|_{L^\infty(0,T; W^{1,q}(\bbsd))} \leq M\}.
\end{equation}
We abuse the notation to write $\Psi \in \mathcal U$ if all the entries $(\Psi)_{ij} \in \mathcal U$.

To analyze the optimal control problem for the mean-field PDE \eqref{A-0}, it is natural to consider an underlying microscopic description in terms of a controlled interacting diffusion on the sphere. Specifically, we introduce an $N$-particle Stratonovich SDE where each agent evolves on $\bbsd$ under a controlled matrix-valued drift $\Psi(t,\cdot)$ and a controlled alignment interaction scaled by $\kappa(t,\cdot)$, together with isotropic spherical noise of intensity $\sigma$: 
\begin{equation}\label{A-2}
d x_i = \Psi(t,x_i)  x_i dt +   \frac{\kappa(t,x_i)}{N}\sum_{k=1}^N \mathbb{P}(x_i) x_k dt + \sqrt{2\sigma} \mathbb{P}(x_i)\circ dW_i,\quad i\in [N],
\end{equation}
where $\{W_i\}_{i=1}^N$ are independent standard  Brownian motions in $\bbr^d$. The noise term $\mathbb{P}(x_i)\circ dW_i$ is understood in the Stratonovich sense  which guarantees the positive invariance of the unit sphere $\bbsd$ with respect to the microscopic flow \eqref{A-2}. In the mean-field limit $N\to\infty$, the empirical measure associated with \eqref{A-2} is expected to converge to the (unique) solution of the mean-field PDE \eqref{A-0}. This stochastic mean-field limit provides the rigorous link between the microscopic controlled dynamics and the macroscopic PDE, thereby justifying the use of \eqref{A-2} as a consistent finite-dimensional approximation for the optimal control problem for \eqref{A-0}.

To make the connection between \eqref{A-2} and \eqref{A-0}, we pass from the stochastic interacting particle system to its equivalent formulation, called Liouville equation that governs the time evolution for the joint law $f^N(t,\cdot)$ of $\{x_i(t)\}_{i=1}^N$ on $(\bbsd)^N$:
\begin{equation} \label{A-3}
\pa_t f^N + \sum_{i=1}^N \nabla_{\omega_i} \cdot\lt(\lt(\Psi(t,\omega_i) \omega_i+ \frac{\kappa(t,\omega_i)}{N}\sum_{k=1}^N \bbp(\omega_i)\omega_k \rt)f^N \rt) = \sigma\sum_{i=1}^N \Delta_{\omega_i} f^N
\end{equation}
where the initial data is given as
\[
f^N\Big|_{t=0} = f_0^{\otimes N} := \prod_{i=1}^N f_0(x_i).
\] 
This reformulation replaces the stochastic dynamics by a deterministic PDE which is particularly convenient in the optimal control problem, since the existence of minimizers and compactness of minimizing sequences can be established through standard PDE arguments. Moreover, the mean-field limit is naturally represented in terms of marginals for the Liouville equation, allowing one to apply relative entropy estimates to obtain quantitative propagation of chaos bounds and to justify passing to the limit along optimal controls.

For the cost functional of the Liouville equation \eqref{A-3}, we need to define the first marginal of $f^N$, denoted by $f^{N;1}[\Psi,\kappa]$ 
\begin{equation} \label{A-3-1}
f^{N;1}[\Psi, \kappa](t,\omega) := \int_{(\bbs^{d-1})^{N-1}}f^N[\Psi, \kappa](t,\omega, \omega_2, \dots, \omega_N)\,d\omega_2\dots d\omega_N.
\end{equation} 
Then, we use the marginal $f^{N;1}$ to construct the cost functional $\tilde{\mathcal J}$ for the microscopic flow \eqref{A-2} as
\begin{align}\label{A-4}
\begin{aligned}
\tilde{\mathcal J} (f^N[\Psi,\kappa],\Psi,\kappa) &:= \frac{\alpha}{2} \int_0^T \int_\bbsd ( f^{N;1}(t,\omega) - z(t,\omega))^2 d\omega dt \\
&\quad + \frac{\beta}{2} \int_{\bbsd} ( f^{N;1}(T,\omega) - z(T,\omega))^2 d\omega \\
&\quad + \frac12 \int_0^T \int_{\bbsd} ( \lambda_1 |\Psi(t,\omega)|^2 + \lambda_2|\kappa(t,\omega)|^2) d\omega dt. 
\end{aligned}
\end{align}
One can easily notice that \eqref{A-4} is structurally the same as \eqref{A-1}. 

% This microscopic viewpoint is important both for modeling and computation: in applications one often designs controls from particle simulations, while in large-population regimes one would like to justify replacing the particle problem by the mean-field PDE control problem.
%
 
Our main results establish a rigorous bridge between finite-dimensional optimal control problem of stochastic interacting agents on the sphere and the corresponding mean-field optimal control problem. First, at the  microscopic level, we formulate the controlled Liouville equation for a system of interacting agents with additive noises and show the existence of optimal controls within an admissible class, relying on the coercivity of the cost functional and weak stability of a solution to the Liouville equation. Second, for any fixed pair of admissible controls, we derive a quantitative stochastic mean-field limit: the one-particle marginal of a solution to the Liouville equation converges to a solution to the controlled mean-field PDE with an explicit convergence rate of an order $N^{-1/2}$, based on relative entropy arguments.   Finally, combining compactness of admissible controls with the quantitative mean-field estimate, we show that the microscopic optimal controls provide consistent approximations to the mean-field optimal controls. Precisely, any weak accumulation point of a sequence of particle-level minimizers is an optimal control for the mean-field problem and the minima of the microscopic cost functional converge to the minimum of the mean-field one as $N\to\infty$. These three assertions are stated below in Theorem \ref{T1.1}.

Here, we remark that our aforementioned framework builds upon on \cite{CWZ}, where the authors studied a mean-field optimal control problem for the Kuramoto model based on the ideas from \cite{FS14, SBB22}.  We extend this Kuramoto framework to the high-dimensional setting and consider the associate mean-field equation posed on the unit sphere. In \cite{CWZ}, the authors used $\mathbb{S}^1 \equiv \mathbb{T}^1$ and employed the propagation of chaos result on $\mathbb{T}^d$ from \cite{JW18} to close the relative entropy estimates. However, since the equivalence $\mathbb{S}^d \equiv \mathbb{T}^d$ holds only for $d=1$, this argument cannot be directly adapted to our case, i.e., $d\ge 2$. The main novelty of our work lies in handling this difficulty; we reformulate our microscopic model as an interacting particle system on $\bbr^d$ with smooth interaction kernels.   From the Liouville equation for the reformulated microscopic system on $\bbr^d$, we can recover the Liouville equation for the microscopic model \eqref{A-2} (see Section \ref{sec:2} for details). This allows us to adopt classical mean-field limit results (see \cite{S91} for example), which in turn enables us to close the relative entropy estimates. To perform all of these without technical gaps, we choose our initial data $f_0$ for the mean-field model \eqref{A-0} to be the restriction to the unit sphere of some $F_0 \in H^{3/2} (\bbr^d)$. Then the trace theorem guarantees $H^1$-regularity for $f_0$, which is necessary for our compactness argument employed for the construction of the mean-field optimal control. We are now ready to state the main results.

\begin{theorem} \label{T1.1}
 Consider $F_0\in H^{3/2}(\bbr^d)$ and let the initial data $f_0$ be its restriction to the unit sphere.  

\begin{enumerate}
\item (Existence of optimal controls for \eqref{A-3}) For any fixed $N$, the cost functional $\tilde{\mathcal J}(f^N[\Psi,\kappa],\Psi,\kappa)$ defined in \eqref{A-4} has a pair of minimizers $(\Psi^N,\kappa^N)\in \mathcal U^2$, i.e., 
\[
\tilde{\mathcal J}(f^N[\Psi^N, \kappa^N], \Psi^N, \kappa^N) = \min_{(\Psi,\kappa) \in\mathcal U^2} \tilde{\mathcal J}(f^N[\Psi, \kappa], \Psi, \kappa) 
\]
where $f^N[\Psi,\kappa]$ is a solution to the Liouville equation \eqref{A-3} for some control functions $\Psi,\kappa\in \mathcal U$. 
\item (Stochastic mean-field limit) Let $f^{N;1}$ be the first marginal defined in \eqref{A-3-1} and $f$ be a solution to \eqref{A-0}. Then, there exists a uniform constant $C>0$ not depending on $N$ such that
\[
\|f^{N;1} - f\|_{L^\infty(0,T; L^1(\bbsd))} \leq \frac{C}{\sqrt N}
\]
for sufficiently large $N\in \bbn$.
\item (Existence of optimal controls for \eqref{A-0}) Any weak accumulation point $(\bar \Psi,\bar \kappa) \in \mathcal U^2$ of $(\Psi^N, \kappa^N)_{N\in \bbn}$ is a minimizer of the cost functional $\mathcal J$ and 
\[
\mathcal J(f[\bar \Psi,\bar \kappa],\bar \Psi,\bar \kappa) = \min_{(\Psi,\kappa)\in \mathcal U^2} \mathcal J ( f[\Psi,\kappa],\Psi,\kappa).
\]
\end{enumerate}
\end{theorem}

\begin{remark}
Since $\bbsd$ is the smooth boundary of the unit ball in $\bbr^d$, the classical Sobolev trace theorem can be applied (see \cite[Theorem 9.4]{LM} for instance). Precisely, for $s>\frac12$, the trace operator extends uniquely to a bounded linear map
\[
\Tr: H^s(\bbr^d) \to H^{s-\frac12} (\bbs^{d-1}) . 
\]
Thus for $F_0|_{\bbs^{d-1}} = f_0$, there exists a constant $C>0$ such that
\[
\|f_0\|_{H^{s-\frac12}(\bbsd)} \leq C \|F_0\|_{H^s(\bbr^d)}.
\]
\end{remark}

The rest of this paper is organized as follows. In Section \ref{sec:2}, we provide several preparatory results for the main results. Specifically, we first reformulate our microscopic and macroscopic models and provide propagation of chaos results. Then, we recall several properties of Sobolev spaces on the unit sphere. Section \ref{sec:3} is devoted to the  proofs of Theorem \ref{T1.1}.

\section{Preliminaries} \label{sec:2}
\setcounter{equation}{0}

In this section, we present reformulated models at microscopic and macroscopic scales, respectively, and recall the properties of Sobolev spaces on the unit sphere.

\subsection{Regularized system}

The mean-field PDE is motivated by the following controlled dynamical multi-agent model:
\begin{equation}\label{ptcle_syst}
dx_i = \Psi(t,x_i) x_i dt + \frac{\kappa(t,x_i)}{N} \sum_{k=1}^N \bbp(x_i)x_k dt + \sqrt{2\sigma} \mathbb{P}(x_i) \circ dW_i
\end{equation}
where $\bbp(x) := \mathbb{I}- \frac{x \otimes x}{|x|^2}$ denotes the $d\times d$ projection  matrix onto the unit sphere $\bbs^{d-1}$. Note that the above can be written in the It$\hat{\textup{o}}$ sense:
\[
\begin{aligned}
dx_i &= \Psi(t,x_i) x_i dt + \frac{\kappa(t,x_i)}{N} \sum_{k=1}^N\bbp(x_i)x_k dt \\
&\quad -\sigma(d-1)\frac{x_i}{|x_i|^2}dt + \sqrt{2\sigma} \mathbb{P}(x_i) dW_i.
\end{aligned}
\]
Here, we first note that the particle system can be recast as a system in $\bbr^d$:
\bq\label{ptcle_syst_rd}
\begin{aligned}
dx_i &= \tilde{\Psi}(t,x_i) \frac{x_i}{|x_i|^2} dt + \frac{\tilde{\kappa}(t,x_i)}{N} \sum_{k=1}^N\bbp(x_i)x_k dt \\
&\quad -\sigma(d-1)\frac{x_i}{|x_i|^2}dt + \sqrt{2\sigma} \mathbb{P}(x_i) dW_i
\end{aligned}
\eq
 where $\tilde{\Psi}(t,x_i):= \Psi\lt(t, \frac{x_i}{|x_i|}\rt)$ and $\tilde\kappa (t,x_i) := \kappa\lt(t, \frac{x_i}{|x_i|}\rt)$.  Then, the corresponding Liouville equation would be of the form

\[\begin{aligned}
\pa_t F^N &+ \sum_{i=1}^N \nabla_{x_i} \lt(\lt(\tilde\Psi(t,x_i) \frac{x_i}{|x_i|^2}+ \frac{\tilde\kappa(t,x_i)}{N}\sum_{k=1}^N \bbp(x_i)x_k -\sigma(d-1) \frac{x_i}{|x_i|^2}\rt)F^N \rt) \\
&= \sigma\sum_{i=1}^N  \sum_{\ell,m=1}^d \pa_{x_i^m}\pa_{x_i^\ell}(\bbp(x_i)_{\ell m}F^ N),
\end{aligned}\]
subject to the initial data $F_0^N$ satisfying $F_0^N |_{(\mathbb{S}^{d-1})^N} = f_0^{\otimes N}$.  Here, for $1 \le \ell, m \le d$, $x_i^\ell$ and $x_i^m$ denote the $\ell$-th and $m$-th component of the vector $x_i \in \bbr^d$, respectively,  and $\mathbb{P}(x_i)_{\ell m}$ denotes the $(\ell, m)$-component of the matrix $\mathbb{P}(x_i)$.

 Once we set $f^N := F^N|_{(\bbs^{d-1})^N}$, we have that $f_0^N = f_0^{\otimes N}$ and it satisfies the following in a weak sense  (see \cite[Section 4.1]{HKLN19} or \cite[Section 2.3]{BCC12} for relevant details)  and in a strong sense if $f^N$ is initially smooth:
\[%\begin{equation} \label{liouville}
\pa_t f^N + \sum_{i=1}^N \nabla_{\omega_i} \cdot\lt(\lt(\Psi(t,\omega_i) \omega_i+ \frac{\kappa(t,\omega_i)}{N}\sum_{k=1}^N \bbp(\omega_i)\omega_k \rt)f^N \rt) = \sigma \sum_{i=1}^N \Delta_{\omega_i} f^N.
%\end{equation}
\]
%Then, the corresponding cost functional is defined as
%\[
%\begin{aligned}
%\mathcal J (f^N[\Psi,\kappa],\Psi,\kappa) &:= \frac{\alpha}{2}\int_0^T \int_{\bbs^{d-1}} |f^{N;1}(t,\omega) - z(t,\omega)|^2\,d\omega dt \\
%&\quad + \frac{\beta}{2}\int_{\bbs^{d-1}}|f^{N;1}(T,\omega) - z(T,\omega)|^2\,d\omega\\
%&\quad + \frac12\int_0^T \int_{\bbs^{d-1}} \lt(\lambda_1 |\Psi(t,\omega)|^2 + \lambda_2 |\kappa(t,\omega)|^2\rt)\,d\omega dt
%\end{aligned}
%\]
%where $f^{N;1}$ denotes the first marginal of $f^N$:
%\[
%f^{N;1}(\Psi, \kappa) := \int_{(\bbs^{d-1})^{N-1}}f^N[\Psi, \kappa](t,\omega_1, \omega_2, \dots, \omega_N)\,d\omega_2\dots d\omega_N.
%\]
Similarly, for the McKean--Vlasov process 
\bq\label{MV_proc}
d{\bar x}_i = \lt(\Psi(t, {\bar x}_i) {\bar x_i} + \kappa(t,{\bar x}_i) \int_{\bbs^{d-1}}\bbp({\bar x}_i) x f_t \,d\omega \rt) + \sqrt{2\sigma} \bbp({\bar x}_i)\circ dW_i,
\eq
which is governed by the law

\[
\p_t f + \nabla_\omega \cdot ( \Psi\omega f + \kappa L[f]f) = \sigma \Delta_\omega f,
\]
we can also recast it as a system or PDE on $\bbr^d$:
\[
d{\bar x}_i = \lt(\Psi(t, {\bar x}_i) \frac{{\bar x_i}}{|{\bar x}_i|^2} + \kappa(t,{\bar x}_i) \int_{\bbr^d}\bbp({\bar x}_i) x F_t \,dx \rt)dt + \sqrt{2\sigma} \bbp({\bar x}_i)\circ dW_i,
\]
and a PDE on $\bbr^d$:
\[
\pa_t F + \nabla_x \cdot \lt[\lt(\Psi \frac{x}{|x|^2} + \kappa \int_{\bbr^d} \bbp(x) x_* F\,dx_*  - \sigma(d-1) \frac{x}{|x|^2}\rt) F\rt] = \sigma \sum_{i,j=1}^d \pa_i \pa_j \lt(\bbp(x)_{ij} F \rt)
\]
subject to initial data $F_0$ which satisfies $F_0|_{\bbs^{d-1}} = f_0$.

However, due to the singularities observed in the system recast in $\bbr^d$,   we do not directly associate \eqref{ptcle_syst} with \eqref{ptcle_syst_rd} to measure the deviation of the microscopic flow \eqref{ptcle_syst}
from the mean-field flow \eqref{MV_proc}. Instead, we introduce some cutoff functions in \eqref{ptcle_syst_rd} and present a regularized system  which still yields the same Liouville equation under restriction to $\bbs^{d-1}$. Motivated by \cite{BCC12}, we define smooth functions $h: \bbr^d \to \bbr$ and $\chi : \bbr^d \to \bbr$ satisfying
\[
h(x) := \left\{\begin{array}{lcl} \frac{1}{|x|^2} & \mbox{if} & |x|\ge \frac12,\\
20 &\mbox{if}& |x|<\frac14. \end{array}\right. \quad \chi(x) := \left\{\begin{array}{lcl} 1 & \mbox{if} &|x| \ge \frac12 \\
0 &\mbox{if} &|x|<\frac14.\end{array}\right.
\]
Then we can also define smooth functions $\varphi_1 : \bbr^d \to \bbr^d\times\bbr^d$, $\varphi_2: \bbr^d \to \bbr^d$ with bounded derivatives satisfying
\[\begin{aligned}
\varphi_1 (x) &=\mathbb{I}-h(x)x\otimes x,\\
\varphi_2 (x) &=h(x) x.
%\sigma_3(v) &= v \quad \mbox{if} \quad |v|\le 2.
\end{aligned}
\]
Note that $\varphi_1(x) = \mathbb{P}(x)$ and $\varphi_2(x) = \frac{x}{|x|^2}$ on $|x| \ge \frac 12$.  We plug  these functions into the microscopic systems and mean-field models recast in $\bbr^d$ and rewrite them as
\begin{equation}\label{ptcle_reg}
\begin{aligned}
dx_i &= \tilde\Psi(t,x_i)\chi(x_i) \varphi_2(x_i ) dt + \frac{\tilde{\kappa}(t,x_i) \chi(x_i)}{N} \sum_{k=1}^N\varphi_1(x_i)x_k dt \\
&\quad -\sigma(d-1)\varphi_2(x_i )dt + \sqrt{2\sigma} \varphi_1(x_i) dW_i,
\end{aligned}
\end{equation}
and 
\bq\label{MV_reg}
\begin{aligned}
d{\bar x}_i &= \lt(\tilde\Psi(t, {\bar x}_i)\chi({\bar x}_i)\varphi_2(\bar x_i) + \kappa(t,{\bar x}_i)\chi({\bar x}_i) \int_{\bbs^{d-1}}\varphi_1({\bar x}_i) x f_t \,d\omega  -\sigma(d-1)\varphi_2({\bar x}_i) \rt)dt \\
&\qquad+ \sqrt{2\sigma} \varphi_1({\bar x}_i) dW_i,
\end{aligned}
\eq
respectively. Note that if the initial data lie on the unit sphere, these systems coincide with \eqref{ptcle_syst} and \eqref{MV_proc}, respectively. Moreover, we can write the corresponding Liouville equation as 
\[\begin{aligned}
\pa_t \tilde{F}^N &+ \sum_{i=1}^N \nabla_{x_i} \lt(\lt(\tilde{\Psi}(t,x_i)\chi(x_i)\varphi_2(x_i)+ \frac{\tilde{\kappa}(t,x_i) \chi(x_i)}{N}\sum_{k=1}^N \varphi_2(x_i)x_k -\sigma(d-1) \varphi_2(x_i)\rt)\tilde{F}^N \rt) \\
&= \sigma\sum_{i=1}^N  \sum_{\ell,m=1}^d \pa_{x_i^m}\pa_{x_i^\ell}(\varphi_1(x_i)_{\ell m}\tilde{F}^N),
\end{aligned}\]
and the law of \eqref{MV_reg} as
\[
\pa_t \tilde{F} + \nabla_x \cdot \lt[\lt(\tilde\Psi \chi \varphi_2(x) + \tilde\kappa  \chi\int_{\bbr^d} \varphi_1(x) x_* \tilde{F}\,dx_*  - \sigma(d-1) \varphi_2(x)\rt) \tilde{F}\rt] = \sigma \sum_{i,j=1}^d \pa_i \pa_j \lt(\varphi_1(x)_{ij} \tilde{F} \rt).
\]
 From the estimates in \cite{HKLN19},   we can also check that $\tilde{F}^N |_{(\bbs^{d-1})^N} = f^N$ and $\tilde{F}|_{\bbs^{d-1}} = f$. At this stage, we can use classical results (see \cite{S91}) to yield the following.

\begin{lemma} 
Suppose $|x_i(0)| = |\bar{x}_i(0)|=1$ for every $i=1,\dots, N$. Then, the solutions to \eqref{ptcle_syst} and \eqref{MV_proc}  coincide with solutions to \eqref{ptcle_reg} and \eqref{MV_reg}, respectively, which emanate from the same initial data. Moreover,  
for given $T>0$, there exists $C>0$ depending on $T$ such that
\[
\bbe\lt[ |x_i - \bar{x}_i|^2 \rt] \le \frac CN, \quad \mbox{for every} \quad 1 \le i \le N \quad \mbox{and} \quad t \le T.
\]
\end{lemma}

%\subsection{Relative entropy estimates}
%
%Using the above, we can obtain relative entropy estimates: Define
% Hence, we need to show
%
%\begin{enumerate}
%\item
%Existence of the optimal control for particle and PDE problems
%
%\item
%Approximation of optimal control for PDE problem via optimal controls at microscopic level
%\end{enumerate}

\subsection{Preliminaries on the unit sphere}

Let $\bbsd \subset \bbr^d$ be the unit sphere equipped with the (normalized) surface measure $d\sigma$, and denote $\Delta_\omega$ by the Laplace-Beltrami operator on $\bbsd$. The spherical harmonics are defined as the eigenfunctions of $(-\Delta_\omega)$. More precisely, for each integer $\ell\in \bbz_+$ called degree, let $\mathcal H_\ell$ be the eigenspace of $(-\Delta_\omega)$ with an eigenvalue $\lambda_\ell:= \ell(\ell+d-2)$. Note that each space $\mathcal H_\ell$ is finite-dimensional, and eigenfunctions corresponding to distinct eigenvalues are orthogonal in $L^2(\bbsd)$. By choosing an orthonormal basis $\{Y_{\ell,k}\}_{k=1}^{N(d,\ell)}$ where $N(d,\ell)$ is the dimension of $\mathcal H_\ell$, the family $\{Y_{\ell,k}\}_{\ell\geq0,1\leq k \leq N(d,\ell)}$ forms a complete orthonormal system in $L^2(\bbsd)$. Consequently,  every $f\in L^2(\bbsd)$ admits an $L^2$-convergent expansion:
\[
f(\omega) = \sum_{\ell=0}^\infty \sum_{k=1}^{N(d,\ell)} a_{\ell,k} Y_{\ell,k}(\omega),\quad a_{\ell,k} := \int_{\bbsd} f(\omega) \overline{Y_{\ell,k}(\omega)}d\sigma(\omega)
\]
and Parseval's identity holds:
\[
\|f\|_{L^2(\bbsd)} ^2 = \sum_{\ell=0}^\infty \sum_{k=1}^{N(d,\ell)} |a_{\ell,k}|^2.
\]
This spectral decomposition provides a representation of $L^2(\bbsd)$ as the Hilbert direct sum of the eigenspaces of $(-\Delta_\omega)$:
\[
L^2(\bbsd) = \bigoplus_{\ell=0}^\infty \mathcal H_\ell(\bbsd).
\]
In addition, we observe
\[
\nabla_\omega f = \sum_{\ell=0}^\infty \sum_{k=1}^{N(d,\ell)} a_{\ell,k} \nabla_\omega Y_{\ell,k} 
\]
which gives
\[
\|\nabla_\omega f\|_{L^2}^2 = \sum_{\ell=0}^\infty \sum_{k=1}^{N(d,\ell)} \lambda_\ell |a_{\ell,k}|^2.
\]
For $s\in \bbr$, the (inhomogeneous) Sobolev space $H^s(\bbsd)$ is defined by 
\[
H^s(\bbsd) := \{ f \in L^2(\bbsd) : \sum_{\ell=0}^\infty \sum_{k=1}^{N(d,\ell)} (1+\lambda_\ell)^s |a_{\ell,k}|^2<\infty\},
\]
equipped with the norm
\[
\|f\|_{H^s{(\bbsd)}}^2 = \sum_{\ell=0}^\infty \sum_{k=1}^{N(d,\ell)} (1+\lambda_\ell)^s |a_{\ell,k}|^2.
\]
Equivalently, it can be written as
\[
\|f\|_{H^s(\bbsd)} = \| (I - \Delta_\omega)^\frac{s}{2} f\|_{L^2(\bbsd)}.
\]
For instance with $s=1$, 
\[
\|f\|_{H^1(\bbsd)}^2 = \|f\|_{L^2(\bbsd)}^2 + \|\nabla_\omega f\|_{L^2(\bbsd)}^2.
\]

On the other hand for $W^{1,q}(\bbsd)$, it would not be natural to use the spherical harmonics. Instead, we recall the definition using local coordinate charts from \cite[Chapter 2]{H}. Let $\{(U_i,\varphi_i)\}_{i=1}^M$ be a fixed finite smooth atlas of $\bbsd$ where each $\varphi_i :U_i \to V_i \subseteq\bbr^d$ is a $C^\infty$-diffeomorphism onto an open set $V_i$, and choose a subordinate $C^\infty$ partition of unity $\{\eta_i\}_{i=1}^M$. For $q\in [1,\infty)$, we define the $L^q$-norm $f$ on $\bbsd$ as
\[
\|f\|_{L^q(\bbsd)}^q := \sum_{i=1}^M \int_{V_i} | (\eta_i f) \circ \varphi_i^{-1}(x)|^q J_i(x) dx,
\]
where $J_i(x)$ denotes the Jacobian density associated with the change of variables $\varphi_i$, i.e., $J_i(x) = \sqrt{\textup{det\,} g_i(x)}$ and $g_i$ is the Riemannian metric tensor of $\bbsd$ represented in the coordinate $x\in V_i$. For $q=\infty$, we set
\[
\|f\|_{L^\infty(\bbsd)} :=\esssup_{\omega\in \bbsd}|f(\omega)|.
\]
Note that this definition coincides with the intrinsic definition
\[
\|f\|_{L^q(\bbsd)}^q = \int_{\bbsd} |f(\omega)|^q d\sigma(\omega),\quad q\in [1,\infty).
\]

We begin by recalling the  Gagliardo-Nirenberg-Sobolev inequality on the unit sphere.	
\begin{lemma}\cite{B93,BDS} \label{GN}
Let $u\in H^1(\bbsd)$. Then, we have
\[
\|\nabla u\|_{L^2(\bbsd)}^2 \geq \frac{d}{p-2} \left( \|u\|_{L^p(\bbsd)}^2 - \|u\|_{L^2(\bbsd)}^2      \right)
\]
for any $p\in [1,2)\cup (2,\infty)$ if $d=1,2$ and for any $p\in [1,2)\cup (2,2^*]$ if $d\geq 3$. Here, $2^*$ is the critical Sobolev exponent, i.e., $2^* = \frac{2d}{d-2}$ for $d\geq3$ and $2^*=\infty$ for $d=1,2$. 
\end{lemma}
 
 We also recall Sobolev-type embedding theorems. 

\begin{lemma}\cite{H} \label{sb}
Let $(M,g)$ be a smooth, compact Riemannian $n$-manifold.
\begin{enumerate}
\item  For $q\geq 1$ and $\lambda\in (0,1)$, if $q\geq \frac{n}{1-\lambda}$, then $W^{1,q}(M) \subset C^\lambda(M)$. 
\item For $q\in [1,n)$ and $p\geq1$ such that $\frac1p > \frac1q - \frac1n$, the embedding of $W^{1,q}(M)$ in $L^p(M)$ is compact. In particular, the embedding of $H^1(M)$in $L^2(M)$ is compact.
\end{enumerate}

\end{lemma}

\begin{proof}
For the proofs, we refer the reader to \cite[Theorem 2.8]{H} and \cite[Theorem 2.9(i)]{H}, respectively.
\end{proof}

Note that    it directly follows from Sobolev embedding theorem for compact manifolds that for all $u\in \mathcal U$, 
\begin{equation} \label{control2}
\|u\|_{L^\infty( 0,T;L^\infty(\bbsd))} \leq C_{Sobolev} M =:\tilde M.
\end{equation}

Finally, we recall the Aubin-Lions lemma.

\begin{lemma}
For $1\leq p,q\leq \infty$,  
define
\[
W := \{ u\in L^p (0,T; H^1(\bbsd)): \partial_t u \in L^q(0,T; H^{-1}(\bbsd))\}.
\]
Then, if $p<\infty$, then the embedding of $W$ into $L^p(0,T; L^2(\bbsd))$ is compact. 
\end{lemma}

\begin{proof}
It follows from Lemma \ref{sb} that $H^1(\bbsd)$ is compactly embedded in $L^2(\bbsd)$ from Lemma \ref{sb} and straightforward calculation shows that $L^2(\bbsd)$ is continuously embedded in $H^{-1}(\bbsd)$. Then, the classical Aubin-Lions lemma gives the desired result. 
\end{proof}

\section{Optimal control problems for the Liouville equation and mean-field PDE} \label{sec:3}
 \setcounter{equation}{0}

In this section, we present the proofs of three assertions in Theorem \ref{T1.1} in the following three subsections, respectively.

\begin{proposition}
Suppose that the initial data $f_0\geq0$ satisfies $f_0 \in H^1(\bbsd)$. Then, for a given $T>0$, \eqref{A-0} has a unique weak solution $f$ in $C(0,T;H^1(\bbsd))$. 
\end{proposition}

\begin{proof}
The proof follows from the similar method developed in \cite[Theorem 2.4]{FL}. 
\end{proof}

\subsection{Proof of Theorem \ref{T1.1}(1)} Let $(\Psi^j,\kappa^j)_{j\in \bbn} \in \mathcal U^2$ be a minimizing sequence of the given minimization problem 
\[
\lim_{j\to\infty} \tilde{\mathcal J}(f^N[\Psi^j,\kappa^j], \Psi^j, \kappa^j) = \inf_{\Psi,\kappa \in \mathcal U} \tilde{\mathcal J} ( f^N[\Psi,\kappa],\Psi,\kappa).
\]
Since the sequence $(\Psi^j,\kappa^j)_{j\in \bbn}$ is uniformly bounded in $ L^\infty(0,T;W^{1,q}(\bbsd))$, there is a weakly-$*$ convergent subsequence (up to relabeling) with a weak limit $(\Psi^*,\kappa^*) \in \mathcal U^2$. Our goal is to show that $(\Psi^*,\kappa^*)$ indeed minimizes the cost functional $\tilde{\mathcal J}$, i.e.,
\[
\lim_{j\to\infty} \tilde{\mathcal J}(f^N[\Psi^j,\kappa^j], \Psi^j, \kappa^j) \geq \tilde{\mathcal J}(f^N[\Psi^*,\kappa^*], \Psi^*, \kappa^*). 
\]
Define the $n$-th marginal of $f^N[\Psi,\kappa]$ for $n=1,2,3$:
\[
\begin{aligned}
&f^{N;1}[\Psi,\kappa](t,\omega) := \int_{(\bbsd)^{N-1}} f^N[\Psi,\kappa](t,\omega,\omega_2,\cdots,\omega_N) d\omega_2\cdots d\omega_N, \\
&f^{N;2}[\Psi,\kappa](t,\omega,\tilde\omega) := \int_{(\bbsd)^{N-2}} f^N[\Psi,\kappa](t,\omega,\tilde \omega,\omega_3,\cdots,\omega_N) d\omega_3\cdots d\omega_N, \\
&f^{N;3}[\Psi,\kappa](t,\omega,\tilde\omega,\hat\omega) := \int_{(\bbsd)^{N-3}} f^N[\Psi,\kappa](t,\omega,\tilde \omega,\hat\omega, \omega_4,\cdots,\omega_N) d\omega_4\cdots d\omega_N.
\end{aligned}
\]
Since our proof crucially relies on the compactness argument, we will show that 
\begin{align} \label{goal}
\begin{aligned}
&\textup{(i) $f^{N;2}[\Psi^j, \kappa^j]$ is uniformly bounded in $L^2(0,T; H^1((\bbsd)^2))$.} \\
&\textup{(ii) $f^{N;1}[\Psi^j, \kappa^j]$ is uniformly bounded in $L^\infty(0,T; H^1(\bbsd))$.}
\end{aligned}
\end{align}
These two statements in \eqref{goal} are shown in Lemma \ref{L3.1} and Lemma \ref{L3.2} below. 
Then, by the Aubin-Lions lemma, we obtain a strongly convergent subsequence (up to relabeling) 
\[
f^{N;1}[\Psi^j, \kappa^j] \to f^{N;1} \quad \textup{in $C(0,T;L^2(\bbsd))$}.
\]
In addition, by taking limits term by term in the weak formulation of the Liouville equation and using the uniqueness of a solution to the Liouville equation, we have
\[
f^{N;1} = f^{N;1}[\Psi^*,\kappa^*].
\]
Finally, we estimate $\tilde{\mathcal J}(f^N[\Psi^j,\kappa^j],\Psi^j,\kappa^j) - \tilde{\mathcal J}(f^N[\Psi^*,\kappa^*],\Psi^*,\kappa^*)$ to show that 
\[
\lim_{j\to\infty} \tilde{\mathcal J}(f^N[\Psi^j,\kappa^j],\Psi^j,\kappa^j) - \tilde{\mathcal J}(f^N[\Psi^*,\kappa^*],\Psi^*,\kappa^*)  \geq 0.
\]
 
Since we have the exchangeability of the particle system, by straightforward calculation, the governing dynamics of the  first marginal $f^{N;1}[\Psi,\kappa]$ is derived as
\begin{align} \label{fN1}
\begin{aligned}
&\p_t f^{N;1}[\Psi,\kappa] - \sigma \Delta_\omega f^{N;1}[\Psi,\kappa] + \nabla_\omega \cdot \Big( \Psi(t,\omega) \omega f^{N;1}[\Psi,\kappa]\Big)  \\
&\quad + \frac{N-1}{N} \nabla_\omega \cdot \left( \kappa(t,\omega) \int_{\bbsd} f^{N;2}[\Psi,\kappa](\omega,\tilde\omega) ( \tilde \omega           - \langle \omega,\tilde \omega\rangle \omega ) d\tilde\omega \right) =0. 
\end{aligned}
\end{align}
Similarly, the second marginal $f^{N;2}[\Psi,\kappa]$ satisfies 
\begin{align} \label{fN2}
\begin{aligned}
&\p_t f^{N;2}[\Psi,\kappa] -\sigma\Big( \Delta_\omega f^{N;2}[\Psi,\kappa] + \Delta_{\tilde\omega} f^{N;2}[\Psi,\kappa]\Big)  \\
&\quad + \nabla_\omega\cdot \Big( \Psi(t,\omega) \omega f^{N;2}[\Psi,\kappa]\Big) + \nabla_{\tilde\omega} \cdot \Big( \Psi(t,\tilde\omega)\tilde\omega f^{N;2}[\Psi,\kappa] \Big)  \\
& \quad+ \frac{N-2}{N} \nabla_\omega \cdot \left( \kappa(t,\omega) \int_{\bbsd} f^{N;3}[\Psi,\kappa](\omega,\tilde\omega,\hat\omega) (\hat\omega - \langle \omega,\hat\omega\rangle \omega) d\hat\omega \right) \\
& \quad + \frac1N \nabla_\omega \cdot \Big( \kappa(t,\omega) (\tilde \omega - \langle \omega,\tilde\omega \rangle \omega) f^{N;2}[\Psi,\kappa] \Big) \\
& \quad +\frac{N-2}{N} \nabla_{\tilde \omega}  \cdot \left( \kappa(t,\tilde \omega) \int_{\bbsd} f^{N;3}[\Psi,\kappa](\omega,\tilde\omega,\hat\omega) (\hat\omega - \langle \tilde \omega,\hat\omega\rangle \tilde \omega) d\hat\omega \right) \\
&\quad + \frac1N \nabla_{\tilde \omega} \cdot \Big( \kappa(t,\tilde \omega) (  \omega - \langle \tilde \omega, \omega \rangle \tilde \omega f^{N;2}[\Psi,\kappa]\Big) =0.
\end{aligned}
\end{align}

\begin{lemma} \label{L3.1} 
 $f^{N;2}[\Psi^j, \kappa^j]$ is uniformly bounded in $L^2(0,T; H^1(\bbsd)^2))$. 
\end{lemma}

\begin{proof}

For $T>0$, our goal is to show that there exists a constant $C>0$ such that 
\[ 
\|f^{N,j;2}(t)\|_{L^2((\bbsd)^2)}^2 \leq C,\quad t\in (0,T].
\]
For simplicity, we write $f^{N,j ; k}:=f^{N;k}[\Psi^j,\kappa^j]$ for $1 \le k \le N$. Then, we multiply $f^{N;j;2}$ with the resulting equation \eqref{fN2} and integrate by parts to find
\begin{align*}
&\frac12\frac{d}{dt}  \int_{\bbsd\times\bbsd} (f^{N,j;2})^2 d\omega d\tilde \omega \\
&\quad   + \sigma\int_{\bbsd\times\bbsd} |\nabla_\omega f^{N,j;2}|^2 d\omega d\tilde \omega + \sigma\int_{\bbsd\times\bbsd}  |\nabla_{\tilde\omega} f^{N,j;2}   |d\tilde\omega  d\omega  \\
& = \int_{\bbsd\times\bbsd} \Psi^j(t,\omega)\omega \cdot \nabla_\omega f^{N,j;2} f^{N,j;2} d\omega d\tilde\omega \\
&\quad  + \int_{\bbsd\times\bbsd} \Psi^j(t,\tilde\omega) \tilde \omega \cdot \nabla_{\tilde\omega} f^{N,j;2} f^{N,j;2} d\omega d\tilde\omega \\
&\quad + \frac{N-2}{N} \int_{\bbsd\times\bbsd} \kappa^j(t,\omega) \left( \int_\bbsd f^{N,j;3}(\omega,\tilde\omega,\hat\omega) (\hat\omega - \langle \omega,\hat\omega\rangle\omega) d\hat\omega \right) \cdot \nabla_\omega f^{N,j;2} d\omega d\tilde\omega \\
&\quad + \frac1N \int_{\bbsd\times\bbsd} \kappa^j(t,\omega) (\tilde \omega- \langle \omega, \tilde\omega\rangle \omega) f^{N,j;2} \cdot \nabla_\omega f^{N,j;2} d\omega d\tilde\omega \\
&\quad +\frac{N-2}{N} \int_{\bbsd\times\bbsd} \kappa^j(t,\tilde \omega) \left( \int_\bbsd f^{N,j;3}(\omega,\tilde\omega,\hat\omega) (\hat\omega - \langle \tilde \omega,\hat\omega\rangle\tilde \omega) d\hat\omega \right) \cdot \nabla_{\tilde\omega} f^{N,j;2} d\tilde \omega d  \omega \\
&\quad + \frac1N \int_{\bbsd\times\bbsd} \kappa^j(t,\omega) (\omega - \tilde \omega,\omega\rangle\tilde\omega) f^{N,j;2} \cdot \nabla_{\tilde\omega} f^{N,j;2}d\tilde\omega d\omega \\
& =: \mathcal I_{11} + \cdots + \mathcal I_{16}.
\end{align*}
By integrating the equation above on $[0,t]$ for any $t\in (0,T]$, we have
\begin{align*}
&\frac12  \int (f^{N,j;2}(t))^2 d\omega d\tilde \omega  - \frac12  \int (f^{N,j;2}(0))^2 d\omega d\tilde \omega  + \sigma\int_0^t \int |\nabla_{(\omega,\tilde\omega)} f^{N,j;2}|^2 d\omega d\tilde\omega ds \\
& = \int_0^t \mathcal I_{11} ds + \cdots + \int_0^t \mathcal I_{16} ds.
\end{align*} 
Below, we estimate $\int_0^t \mathcal I_{1k} ds$ for $k=1,\cdots,6$, respectively. \newline

\noindent $\bullet$ (Estimates for $\int_0^t \mathcal I_{11}ds$): We use H\"older's inequality to find
\begin{align*}
&\int_0^t \mathcal I_{11}ds   = \int_0^t \int_{\bbsd\times\bbsd} \Psi^j(t,\tilde\omega) \tilde \omega \cdot \nabla_{\tilde\omega} f^{N,j;2} f^{N,j;2} d\omega d\tilde\omega ds \\
&\hspace{0.5cm} \leq \|\Psi^j\|_{L^\infty(0,T; L^\infty(\bbsd))} \|f^{N,j;2}\|_{L^2(0,T;L^2((\bbsd)^2))} \|\nabla_\omega f^{N,j;2}\|_{L^2(0,T;L^2((\bbsd)^2))}.
\end{align*}
\noindent $\bullet$ (Estimates for $\int_0^t\mathcal I_{12}ds$): Similar to the first case, 
\begin{align*}
&\int_0^t \mathcal I_{12} ds   = \int_0^t \int_{\bbsd\times\bbsd} \Psi^j(t,\tilde\omega) \tilde \omega \cdot \nabla_{\tilde\omega} f^{N,j;2} f^{N,j;2} d\omega d\tilde\omega ds \\
& \hspace{0.5cm} \leq \|\Psi^j\|_{L^\infty(0,T; L^\infty(\bbsd))} \|f^{N,j;2}\|_{L^2(0,T;L^2((\bbsd)^2))} \|\nabla_{\tilde\omega} f^{N,j;2}\|_{L^2(0,T;L^2((\bbsd)^2))}.
\end{align*}
\noindent $\bullet$ (Estimates for $\int_0^t\mathcal I_{13}ds$): We observe
\begin{align*}
& \left| \int_\bbsd f^{N,j;3}(\omega,\tilde\omega,\hat\omega) (\hat\omega - \langle \omega,\hat\omega\rangle\omega) d\hat\omega \right|    \leq 2 \int_{\bbsd} f^{N,j;3}(\omega,\tilde\omega,\hat\omega) d\hat \omega = 2 f^{N,j;2}.
\end{align*}
Thus, we get
\begin{align*}
&\int_0^t \mathcal I_{13} ds  \\
& = \frac{N-2}{N}\int_0^t  \int_{\bbsd\times\bbsd} \kappa^j(t,\omega) \left( \int_\bbsd f^{N,j;3}(\omega,\tilde\omega,\hat\omega) (\hat\omega - \langle \omega,\hat\omega\rangle\omega) d\hat\omega \right) \cdot \nabla_\omega f^{N,j;2} d\omega d\tilde\omega ds  \\
& \leq \frac{2(N-2)}{N}  \|\kappa^j \|_{L^\infty(0,T;L^\infty(\bbsd))} \|f^{N,j;2}\|_{L^2(0,T;L^2((\bbsd)^2))} \|\nabla_{\omega} f^{N,j;2}\|_{L^2(0,T;L^2((\bbsd)^2))}.
\end{align*}
\noindent $\bullet$ (Estimates for $\int_0^t\mathcal I_{14}ds$): We see
\begin{align*}
\int_0^t \mathcal I_{14} ds & =\frac1N \int_{\bbsd\times\bbsd} \kappa^j(t,\omega) (\tilde \omega- \langle \omega, \tilde\omega\rangle \omega) f^{N,j;2} \cdot \nabla_\omega f^{N,j;2} d\omega d\tilde\omega  \\
& \leq  \frac2N \|\kappa^j \|_{L^\infty(0,T;L^\infty(\bbsd))} \|f^{N,j;2}\|_{L^2(0,T;L^2((\bbsd)^2))} \|\nabla_{\omega} f^{N,j;2}\|_{L^2(0,T;L^2((\bbsd)^2))}.
\end{align*}
\noindent $\bullet$ (Estimates for $\int_0^t\mathcal I_{15}ds$): Similar to the third case, 
\begin{align*}
\int_0^t\mathcal I_{15}ds \leq \frac{2(N-2)}{N}  \|\kappa^j \|_{L^\infty(0,T;L^\infty(\bbsd))} \|f^{N,j;2}\|_{L^2(0,T;L^2((\bbsd)^2))} \|\nabla_{\tilde\omega} f^{N,j;2}\|_{L^2(0,T;L^2((\bbsd)^2))}.
\end{align*}
\noindent $\bullet$ (Estimates for $\int_0^t\mathcal I_{16}ds$): Similar to the fourth case, 
\begin{align*}
\int_0^t\mathcal I_{16}ds &=\frac1N \int_{\bbsd\times\bbsd} \kappa^j(t,\omega) (\omega - \tilde \omega,\omega\rangle\tilde\omega) f^{N,j;2} \cdot \nabla_{\tilde\omega} f^{N,j;2}d\tilde\omega d\omega \\
& \leq \frac2N \|\kappa^j \|_{L^\infty(0,T;L^\infty(\bbsd))} \|f^{N,j;2}\|_{L^2(0,T;L^2((\bbsd)^2))} \|\nabla_{\tilde\omega} f^{N,j;2}\|_{L^2(0,T;L^2((\bbsd)^2))}.                
\end{align*}
To this end,  we collect all estimates to find 
\begin{align} \label{C-9}
\begin{aligned}
&\frac12  \int_{\bbsd\times \bbsd} (f^{N,j;2}(t))^2 d\omega d\tilde \omega  - \frac12  \int_{\bbsd\times \bbsd} (f^{N,j;2}(0))^2 d\omega d\tilde \omega  \\
& \quad + \sigma\int_0^t \int_{\bbsd\times \bbsd} |\nabla_{(\omega,\tilde\omega)} f^{N,j;2}|^2 d\omega d\tilde\omega \\
&\quad \leq 2\tilde M \|f^{N,j;2}\|_{L^2(0,T;L^2((\bbsd)^2))} \|\nabla_{\omega} f^{N,j;2}\|_{L^2(0,T;L^2((\bbsd)^2))} \\ 
&\quad \quad+2\tilde M \|f^{N,j;2}\|_{L^2(0,T;L^2((\bbsd)^2))} \|\nabla_{\tilde\omega} f^{N,j;2}\|_{L^2(0,T;L^2((\bbsd)^2))} \\
& \quad \leq \frac{4\tilde M }{\sigma} \|f^{N,j;2}\|_{L^2(0,T;L^2((\bbsd)^2))}^2 + \frac \sigma2 \|\nabla_{(\omega,\tilde\omega)} f^{N,j;2}\|_{L^2(0,T;L^2((\bbsd)^2))}^2.
\end{aligned}
\end{align}
Finally, Gr\"onwall's inequality gives the desired constant $C>0$ such that
\[
\|f^{N,j;2}(t)\|_{L^2((\bbsd)^2)}^2 \leq C,\quad t\in (0,T].
\]
In addition for $h=h(\omega,\tilde\omega)$, 
\begin{align*}
|\nabla_{(\omega,\tilde\omega)} h|^2 = |\nabla_{\omega}h|^2 + |\nabla_{\tilde \omega}h|^2,
\end{align*}
we also obtain  from \eqref{C-9} that 
\begin{equation} \label{C-10}
 \|\nabla_{\omega} f^{N,j;2}\|_{L^2(0,T;L^2((\bbsd)^2))}^2 +  \|\nabla_{\tilde \omega} f^{N,j;2}\|_{L^2(0,T;L^2((\bbsd)^2))}^2 \leq C.
\end{equation}
This gives the desired uniform boundedness in $L^2(0,T; H^1( (\bbsd)^2))$. 
\end{proof}

Next, we are concerned with the uniform boundedness of $f^{N,j;1}$ in $L^\infty(0,T; H^1(\bbsd))$. 
\begin{lemma} \label{L3.2}
$f^{N;1}[\Psi^j,\kappa^j]$ is uniformly bounded in $L^\infty(0,T; H^1(\bbsd))$.
\end{lemma}

\begin{proof}
(i) First, we show that $f^{N;1}[\Psi^j,\kappa^j](t)$ is uniformly bounded in $L^2(\bbsd)$. As we did in Lemma \ref{L3.1},  we again write $f^{N,j;k}:=f^{N;k}[\Psi^j,\kappa^j]$ for $1 \le k \le N$ for simplicity. For this, we multiply $f^{N,j;1}$ with \eqref{fN1} and integrate by parts to find
\begin{align*}
&\frac12\frac{d}{dt} \int_\bbsd (f^{N,j;1})^2 d\omega +\sigma\int_\bbsd |\nabla_\omega f^{N,j;1}|^2 d\omega \\
&=\int_\bbsd \Psi^j \omega \cdot \nabla_\omega f^{N,j;1} f^{N,j;1} d\omega \\
&\quad + \frac{N-1}{N} \kappa^j \left( \int_\bbsd f^{N;2}[\Psi^j, \kappa^j] (\omega,\tilde\omega) (\tilde \omega - \langle \omega,\tilde\omega\rangle \omega d\tilde \omega \right)\cdot \nabla_\omega f^{N,j;1} d\omega .
\end{align*}
By the similar argument in Lemma \ref{L3.1}, we have
\begin{align*}
&\frac12\frac{d}{dt} \int_\bbsd (f^{N,j;1})^2 d\omega +\sigma\int_\bbsd |\nabla_\omega f^{N,j;1}|^2 d\omega \\
&\leq \|\Psi^j\|_{L^\infty(0,T;L^\infty(\bbsd))} \|\nabla_\omega f^{N,j;1} \|_{L^2(\bbsd)} \|f^{N,j;1}\| _{L^2(\bbsd)} \\
&\quad+ \frac{2(N-1)}{N} \|\kappa^j\|_{L^\infty(0,T; L^\infty(\bbsd))} \|\nabla_\omega f^{N,j;1} \|_{L^2(\bbsd)} \|f^{N,j;1}\| _{L^2(\bbsd)} \\
&\leq 3\tilde M\|\nabla_\omega f^{N,j;1} \|_{L^2(\bbsd)} \|f^{N,j;1}\| _{L^2(\bbsd)} \\
&\leq \frac{9\tilde M^2}{2\sigma} \|f^{N,j;1}\| _{L^2(\bbsd)} ^2 + \frac \sigma2 \|\nabla_\omega f^{N,j;1} \|_{L^2(\bbsd)}^2.
\end{align*}
Then, Gr\"onwall's inequality gives the desired constant $C>0$ such that
\begin{equation} \label{C-40}
\| f^{N,j;1}(t)\| _{L^2(\bbsd)} ^2 \leq C,\quad t\in (0,T].
\end{equation}
(ii) Our next goal is to show that $\nabla_\omega f^{N,j;1}(t)$ is uniformly bounded in $L^2(\bbsd)$: 
\[
\|\nabla_\omega f^{N,j;1}(t)\|_{L^2(\bbsd)}^2 \leq C,\quad t\in (0,T].
\]
%For this, we  expand the two divergence terms in \eqref{fN1}. First, 
%\begin{align*}
%\nabla_\omega \cdot  (\Psi^j(t,\omega)\omega f^{N,j;1} )  = \nabla_\omega \cdot (\Psi^j(t,\omega)\omega) f^{N,j;1} + \Psi^j(t,\omega)\omega \cdot \nabla_\omega f^{N,j;1}.
%\end{align*}
%and we have
%\[
%|\nabla_\omega \cdot (\Psi^j(t,\omega)\omega) |\leq \sqrt{d-1}|\nabla_\omega \Psi^j(t,\omega)|.
%\]
%Second, 
%\begin{align*}
% \nabla_\omega \cdot \left( \kappa^j(t,\omega) F(\omega) \right) = \kappa^j(t,\omega) \nabla_\omega \cdot F(\omega) + \nabla_\omega \kappa^j(t,\omega) \cdot F(\omega)
%\end{align*} 
%where
%\[
%F(\omega) := \int_{\bbsd} f^{N,j;2}(\omega,\tilde\omega) ( \tilde \omega           - \langle \omega,\tilde \omega\rangle \omega ) d\tilde\omega.
%\]
%Note that 
%\begin{align*}
%\nabla_\omega \cdot F(\omega) & = \int_{\bbsd}\nabla_\omega f^{N,j;2}(\omega,\tilde\omega) \cdot (\tilde \omega -\langle \omega,\tilde\omega\rangle\omega)d\tilde\omega \\
%&\quad  + \int_\bbsd f^{N,j;2}(\omega,\tilde\omega)  \underbrace{ \nabla_\omega \cdot (\tilde\omega-\langle\omega,\tilde\omega\rangle\omega)}_{= -(d-1)\langle\omega,\tilde\omega\rangle} d\tilde \omega \\
%&\leq 2 \left|\int_\bbsd |\nabla_\omega f^{N,j;2}|^2 d\tilde\omega\right|^\frac12 + (d-1)f^{N;1}.
%\end{align*}
%
We multiply $(-\Delta_\omega f^{N,j;1})$ with \eqref{fN1} and integrate by parts to obtain 
\begin{align*}
&\frac12\frac{d}{dt} \int_\bbsd |\nabla_\omega f^{N,j;1}|^2 d\omega +\sigma\int_\omega |\Delta_\omega f^{N,j;1}|^2 \d\omega \\
& = \int_{\bbsd} \nabla_\omega \cdot (\Psi^j[t,\omega]\omega) f^{N,j;1} (-\Delta_\omega f^{N,j;1}) d\omega \\
&\quad + \int_\bbsd \Psi^j(t,\omega) \omega \cdot \nabla_\omega f^{N,j;1} (-\Delta_\omega f^{N,j;1}) d\omega \\
& \quad + \frac{N-1}{N} \int_\bbsd \nabla_\omega \kappa^j(t,\omega) \cdot \left( \int_\bbsd f^{N,j;2}(\omega,\tilde\omega) (\tilde \omega -\langle \omega,\tilde\omega\rangle\omega) d\tilde\omega     \right) (-\Delta_\omega f^{N,j;1}) d\omega \\
&\quad +\frac{N-1}{N} \int_\bbsd \kappa^j(t,\omega) \nabla_\omega \cdot \left( \int_\bbsd f^{N,j;2}(\omega,\tilde\omega) (\tilde \omega -\langle \omega,\tilde\omega\rangle\omega) d\tilde\omega     \right) (-\Delta_\omega f^{N,j;1}) d\omega  \\
&=: \mathcal I_{21} + \cdots + \mathcal I_{24}.
\end{align*}
By integrating the equation above on $[0,t]$ for any $t\in(0,T]$, we have
\begin{align*}
&\frac12 \int_\bbsd (f^{N,j;1}(t))^2 d\omega - \frac12 \int_\bbsd (f^{N,j;1})(0))^2d\omega + \sigma\int_0^t\int_{\bbsd} |\Delta_\omega f^{N,j;1}|^2 \,d\omega ds  \\
& = \int_0^t \mathcal I_{21} ds + \cdots + \int_0^t \mathcal I_{24}ds.
\end{align*}
Below, we estimate $\int_0^t \mathcal I_{2k} ds$ for $k=1,\cdots,4$, respectively. \newline

\noindent  $\bullet$ (Estimates for $\int_0^t \mathcal I_{21} ds$): First, we observe  
\[
|\nabla_\omega \cdot (\Psi^j(t,\omega)\omega) |\leq \sqrt{d-1}|\nabla_\omega \Psi^j(t,\omega)|.
\]
Thus, we use H\"older's inequality to obtain
\begin{align*}
\int_0^t \mathcal I_{21} ds& =\int_0^t  \int_{\bbsd} \nabla_\omega \cdot (\Psi^j[t,\omega]\omega) f^{N,j;1} (-\Delta_\omega f^{N,j;1}) d\omega ds \\
&\leq \sqrt{d-1}\|\Psi^j\|_{L^\infty(0,T;L^q(\bbsd))} \|f^{N,j;1}\|_{L^2(0,T; L^\frac{2q}{q-2}(\bbsd))} \|\Delta_\omega f^{N,j;1}\|_{L^2(0,T; L^2(\bbsd))}.
\end{align*}

\noindent $\bullet$ (Estimates for $\int_0^t \mathcal I_{22} ds$): We observe
\begin{align*}
\int_0^t \mathcal I_{22} ds& =\int_0^t  \int_\bbsd \Psi^j(t,\omega) \omega \cdot \nabla_\omega f^{N,j;1} (-\Delta_\omega f^{N,j;1}) d\omega ds \\
& \leq \|\Psi^j\|_{L^\infty(0,T;L^\infty(\bbsd))} \|\nabla f^{N,j;1}\|_{L^2(0,T; L^2(\bbsd))} \|\Delta_\omega f^{N,j;1}\|_{L^2(0,T; L^2(\bbsd))}.
\end{align*}

\noindent $\bullet$ (Estimates for $\int_0^t \mathcal I_{23} ds$): Note that
\begin{align*}
\left| \int_\bbsd f^{N,j;2} (\omega,\tilde\omega) (\tilde \omega -\langle \omega,\tilde\omega\rangle\omega) d\tilde \omega \right| \leq 2f^{N,j;1}.
\end{align*}
Thus, we obtain
\begin{align*}
& \int_0^t \mathcal I_{23} ds  \\
& = \frac{N-1}{N}\int_0^t         \int_\bbsd \kappa^j(t,\omega) \nabla_\omega \cdot \left( \int_\bbsd f^{N,j;2}(\omega,\tilde\omega) (\tilde \omega -\langle \omega,\tilde\omega\rangle\omega) d\tilde\omega     \right) (-\Delta_\omega f^{N,j;1}) d\omega ds      \\
&\quad \leq  2\|\nabla_\omega \kappa^j \|_{L^\infty(0,T;L^q(\bbsd))}  \|f^{N,j;1}\|_{L^2(0,T;L^\frac{2q}{q-2}(\bbsd))}  \|\Delta_\omega f^{N,j;1}\|_{L^2(0,T; L^2(\bbsd))}.
\end{align*}

\noindent $\bullet$ (Estimates for $\int_0^t \mathcal I_{24} ds$): Denote
\[
F(\omega) := \int_{\bbsd} f^{N,j;2}(\omega,\tilde\omega) ( \tilde \omega           - \langle \omega,\tilde \omega\rangle \omega ) d\tilde\omega.
\]
Note that 
\begin{align*}
\nabla_\omega \cdot F(\omega) & = \int_{\bbsd}\nabla_\omega f^{N,j;2}(\omega,\tilde\omega) \cdot (\tilde \omega -\langle \omega,\tilde\omega\rangle\omega)d\tilde\omega \\
&\quad  + \int_\bbsd f^{N,j;2}(\omega,\tilde\omega)  \underbrace{ \nabla_\omega \cdot (\tilde\omega-\langle\omega,\tilde\omega\rangle\omega)}_{= -(d-1)\langle\omega,\tilde\omega\rangle} d\tilde \omega \\
&\leq 2 \left|\int_\bbsd |\nabla_\omega f^{N,j;2}|^2 d\tilde\omega\right|^\frac12 + (d-1)f^{N,j;1}.
\end{align*}
Then, we get 
\begin{align*}
\| \nabla_\omega \cdot F(\omega) \|_{L^2(\bbsd)}^2 &= \int_\bbsd |\nabla_\omega \cdot F(\omega)|^2d\omega \\
&\leq 8\int_{\bbsd\times \bbsd} |\nabla_\omega  f^{N,j;2}|^2 d\tilde\omega d\omega + 2(d-1)^2 \int_\bbsd (f^{N,j;1})^2 d\omega  \\
& = 8 \|\nabla_\omega f^{N,j;2} \|_{L^2(\bbsd\times\bbsd)}^2 + 2(d-1)^2 \|f^{N,j;1}\|_{L^2(\bbsd)}^2
\end{align*}
which yields 
\begin{align*}
\| \nabla_\omega \cdot F(\omega) \|_{L^2(\bbsd)} & \leq 2\sqrt2 \|\nabla_\omega f^{N,j;2} \|_{L^2(\bbsd\times\bbsd)} + \sqrt2(d-1) \|f^{N,j;1}|\|_{L^2(\bbsd)} \\
&\leq C + \sqrt{2}(d-1) \|f^{N,j;1}|\|_{L^2(\bbsd)} \\
&\leq C.
\end{align*} 
Here, we used \eqref{C-10} in Lemma \ref{L3.1} and \eqref{C-40}. Hence, we obtain
\begin{align*}
&\int_0^t \mathcal I_{24} ds \\
& =\frac{N-1}{N} \int_0^t  \int_\bbsd \kappa^j(t,\omega) \nabla_\omega \cdot \left( \int_\bbsd f^{N,j;2}(\omega,\tilde\omega) (\tilde \omega -\langle \omega,\tilde\omega\rangle\omega) d\tilde\omega     \right) (-\Delta_\omega f^{N,j;1}) d\omega ds \\
& = \frac{N-1}{N}\int_0^t \int_\bbsd \kappa^j(t,\omega) \nabla_\omega \cdot F(\omega) (-\Delta_\omega f^{N,j;1}) d\omega ds \\
& \leq \|\kappa^j\|_{L^\infty(0,T;L^\infty(\bbsd))} \|\nabla_\omega \cdot F(\omega) \|_{L^2(0,T; L^2(\bbsd))} \|\Delta_\omega f^{N,j;1}\|_{L^2(0,T; L^2(\bbsd))} \\
& \leq C\|\kappa^j\|_{L^\infty(0,T;L^\infty(\bbsd))}\|\Delta_\omega f^{N,j;1}\|_{L^2(0,T; L^2(\bbsd))}.
\end{align*}

So far, we have
\begin{align*}
& \frac12 \|\nabla_\omega f^{N,j;1}(t)\|_{L^2(\bbsd)} - \frac12 \|\nabla_\omega f^{N,j;1}(0)\|_{L^2(\bbsd)} + \sigma\|\Delta_\omega f^{N,j;1} \|_{L^2(0,T;L^2(\bbsd))}^2 \\
& \leq \sqrt{d-1}\|\Psi^j\|_{L^\infty(0,T;L^q(\bbsd))} \|f^{N,j;1}\|_{L^2(0,T; L^\frac{2q}{q-2}(\bbsd))} \|\Delta_\omega f^{N,j;1}\|_{L^2(0,T; L^2(\bbsd))} \\
&\quad + \|\Psi^j\|_{L^\infty(0,T;L^\infty(\bbsd))} \|\nabla f^{N,j;1}\|_{L^2(0,T; L^2(\bbsd))} \|\Delta_\omega f^{N,j;1}\|_{L^2(0,T; L^2(\bbsd))} \\
&\quad + 2\|\nabla_\omega \kappa^j \|_{L^\infty(0,T;L^q(\bbsd))}  \|f^{N,j;1}\|_{L^2(0,T;L^\frac{2q}{q-2}(\bbsd))} \| \|\Delta_\omega f^{N,j;1}\|_{L^2(0,T; L^2(\bbsd))} \\
&\quad + C\|\kappa^j\|_{L^\infty(0,T;L^\infty(\bbsd))}\|\Delta_\omega f^{N,j;1}\|_{L^2(0,T; L^2(\bbsd))} \\
&\leq C \|f^{N,j;1}\|_{L^2(0,T; L^\frac{2q}{q-2}(\bbsd))} \| \Delta_\omega f^{N,j;1}\|_{L^2(0,T; L^2(\bbsd))} \\
& \quad + C\|\nabla f^{N,j;1}\|_{L^2(0,T; L^2(\bbsd))} \|\Delta_\omega f^{N,j;1}\|_{L^2(0,T; L^2(\bbsd))} \\
& \quad + C \| \Delta_\omega f^{N,j;1}\|_{L^2(0,T; L^2(\bbsd))}.
\end{align*}

Now, it suffices to estimate $\|f^{N,j;1}(t) \|_{L^\frac{2q}{q-2}(\bbsd)}^2$. We use Lemma \ref{GN} with $p=\frac{2q}{q-2}$ to see that 
\[
\|f^{N,j;1}(t)\|_{L^\frac{2q}{q-2}(\bbsd)}^2 \leq \|f^{N,j;1}(t) \|_{L^2(\bbsd)}^2 + \frac{p-2}{d} \|\nabla_\omega f^{N,j;1}(t)\|_{L^2(\bbsd)}^2
\]
Here, if $d\geq3$, then we need to assume that $q\geq d$ so that
\[
\frac{2q}{q-2} \leq 2^*=\frac{2d}{d-2}.
\]
Thus, we have
\begin{align} \label{C-60}
\begin{aligned}
\|f^{N,j;1}\|_{L^2(0,T;L^\frac{2q}{q-2}(\bbsd))}^2 &= \int_0^t \|f^{N,j;1}(t)\|_{L^\frac{2q}{q-2}(\bbsd)}^2 ds \\
&  \leq \int_0^t \|f^{N,j;1}(t)\|_{L^2(\bbsd)}^2 ds \\
&\quad + \frac{4}{d(q-2)} \int_0^t \|\nabla_\omega f^{N,j;1}\|_{L^2(\bbsd)}^2 \\
& \leq C( 1+ \| \nabla_\omega f^{N,j;1}\|_{L^2(0,T;L^2(\bbsd))}^2)
\end{aligned}
\end{align}
where we used \eqref{C-40}. To this end, we use H\"older's inequality to obtain 
\begin{align*}
 &\frac12 \|\nabla_\omega f^{N,j;1}(t)\|_{L^2(\bbsd)} - \frac12 \|\nabla_\omega f^{N,j;1}(0)\|_{L^2(\bbsd)} + \sigma\|\Delta_\omega f^{N,j;1} \|_{L^2(0,T;L^2(\bbsd))}^2  \\
 & \leq \frac{C^2}{2\veps} \|f^{N,j;1}\|_{L^2(0,T; L^\frac{2q}{q-2}(\bbsd))}^2 + \frac\veps2 \| \Delta_\omega f^{N,j;1}\|_{L^2(0,T; L^2(\bbsd))} ^2 \\
 &\quad + \frac{C^2}{2\veps} \|\nabla_\omega f^{N,j;1}\|_{L^2(0,T; L^2(\bbsd))} ^2 +  \frac\veps2 \| \Delta_\omega f^{N,j;1}\|_{L^2(0,T; L^2(\bbsd))} ^2  \\
 & \quad + \frac{C^2}{2\veps} +\frac\veps2 \| \Delta_\omega f^{N,j;1}\|_{L^2(0,T; L^2(\bbsd))} ^2 \\
 & =    \frac{C^2}{2\veps} + \frac{C^2}{2\veps} \|f^{N,j;1}\|_{L^2(0,T; L^\frac{2q}{q-2}(\bbsd))}^2  \\
 &\quad + \frac{C^2}{2\veps} \|\nabla_\omega f^{N,j;1}\|_{L^2(0,T; L^2(\bbsd))} ^2 +\frac{3\veps}{2} \| \Delta_\omega f^{N,j;1}\|_{L^2(0,T; L^2(\bbsd))} ^2 \\
 & \leq C + C  \|\nabla_\omega f^{N,j;1}\|_{L^2(0,T; L^2(\bbsd))} ^2 + \frac \sigma2  \| \Delta_\omega f^{N,j;1}\|_{L^2(0,T; L^2(\bbsd))} ^2 
\end{align*} 
by choosing $\veps = \frac \sigma3$ and using \eqref{C-60}. By applying Gr\"onwall's inequality to find
\begin{align*}
\|\nabla_\omega f^{N,j;1}\|_{L^2(\bbsd)} ^2\leq C,\quad t\in (0,T]
\end{align*}
uniformly in $j$. Hence, together with \eqref{C-40}, we conclude that the sequence $(f^{N,j;1})_{j\in\bbn}$ is uniformly bounded in $L^\infty(0,T; H^1(\bbsd))$.
\end{proof}

Now, we are ready to provide the proof of Theorem \ref{T1.1}(1). \newline

\noindent \textit{Proof of Theorem \ref{T1.1}(1):}  In order to show that $(\partial_t  f^{N,j;1})_{j\in \bbn} \in L^2(0,T; H^{-1}(\bbsd))$, we take the inner product \eqref{fN1} with $\phi \in H^1(\bbsd)$ and use the regularity assumptions \eqref{A-1-2} and \eqref{control2} for control functions in  to find
\begin{align*}
\langle \partial_t f^{N,j;1} ,\varphi\rangle & \leq \|\nabla_\omega f^{N,j;1}\|_{L^2} \|\varphi\|_{H^1}  \\
&\quad + (\|\Psi^j(t)\|_{L^\infty(\bbsd)} + \|\kappa^j(t)\|_{L^\infty(\bbsd)} ) \|f^{N,j;1} \|_{L^2(\bbsd)} \|\varphi\|_{H^1(\bbsd)}.
\end{align*} 
We invoke the Aubin-Lions lemma to obtain a strongly convergent subsequence (up to relabeling) 
\[
f^{N,j;1} \to f^{N;1} \quad \textup{in $C(0,T;L^2(\bbsd))$.}
\]
Then, by the weak formulation of \eqref{fN1} whose weak solution is unique, we conclude that
\[
f^{N;1} = f^{N;1}[\Psi^*,\kappa^*]. 
\]
Finally, we show that $(\Psi^*,\kappa^*)$ indeed minimizes the cost functional $\tilde{\mathcal J}$ in \eqref{A-4}:
\[
\lim_{j\to\infty} \tilde{\mathcal J}(f^N[\Psi^j,\kappa^j], \Psi^j, \kappa^j) \geq \tilde{\mathcal J}(f^N[\Psi^*,\kappa^*], \Psi^*, \kappa^*). 
\]
Although the proof directly follows from \cite{CWZ}, we provide it for completeness. We write
\begin{align*}
&J_N(f^N[\Psi^j,\kappa^j], \Psi^j, \kappa^j) - \tilde{\mathcal J}(f^N[\Psi^*,\kappa^*], \Psi^*, \kappa^*) \\
& = \frac{\alpha}{2} \int_0^T \int_\bbsd |f^{N,j;1}(t,\omega) - z(t,\omega)|^2 - |f^{N;1}[\Psi^*,\kappa^*] - z(t,\omega)|^2 d\omega dt \\
&\quad + \frac{\beta}{2}\int_\bbsd |f^{N,j;1}(T,\omega) - z(T,\omega)|^2 - |f^{N;1}[\Psi^*,\kappa^*] - z(T,\omega)|^2 d\omega \\
&\quad + \frac12 \int_0^T \int_\bbsd \lambda_1 ( |\Psi^j(t,\omega)|^2 - |\Psi^*(t,\omega)|^2) + \lambda_2( |\kappa^j(t,\omega)|^2 - |\kappa^*(t,\omega)|^2) d\omega dt  \\
& =: \mathcal I_{31} + \mathcal I_{32} + \mathcal I_{33}. 
\end{align*}
Below, we estimate $\mathcal I_{3k}$ for $k=1,2,3$, respectively. \newline

\noindent $\bullet$ (Estimates for $\mathcal I_{31}$): We have
\begin{align*}
\mathcal I_{31}  & = \frac{\alpha}{2} \int_0^T \int_\bbsd | f^{N,j;1}(t,\omega) - z(t,\omega)|^2 - |f^{N;1}[\Psi^*,\kappa^*] - z(t,\omega)|^2 d\omega dt \\
&\leq \frac{\alpha}{2} \| f^{N,j;1} - f^{N;1}[\Psi^*,\kappa^*]\|_{L^2(0,T; L^2(\bbsd))}  \\
&\quad \times ( \|f^{N,j;1}-z\|_{L^2(0,T;L^2(\bbsd))} + \| f^{N;1}[\Psi^*,\kappa^*] - z\|_{L^2(0,T;L^2(\bbsd))} )
\end{align*} 
Since \eqref{C-40} holds for any $j\in \bbn$ and
\begin{align*}
\|f^{N;1}[\Psi^*,\kappa^*]\|_{L^2(\bbsd)}  \leq \|f^{N;1}[\Psi^*,\kappa^*] - f^{N,j;1}\|_{L^2(\bbsd)} + \|f^{N,j;1}\|_{L^2(\bbsd)}  \leq C,
\end{align*}
we see that 
\[
\lim_{j\to\infty} \mathcal I_{31} =0.
\]
$\bullet$ (Estimates for $\mathcal I_{32}$): Similar to $\mathcal I_{31}$,  we see 
\begin{align*}
\mathcal I_{32} & =  \frac{\beta}{2}\int_\bbsd | f^{N,j;1}(T,\omega) - z(T,\omega)|^2 - f^{N;1}[\Psi^*,\kappa^*] - z(T,\omega)|^2 d\omega \\
&\leq \frac{\beta}{2} \| f^{N,j;1} (T) - f^{N;1}[\Psi^*,\kappa^*](T)\|_{L^2(\bbsd)} \\
&\quad \times ( \| f^{N,j;1}(T)  -z(T) \|_{L^2(\bbsd)}  +  \| f^{N;1}[\Psi^*,\kappa^*](T)  -z(T) \|_{L^2(\bbsd)} ).
\end{align*}
Then, we directly find that
\[
\lim_{j\to\infty} \mathcal I_{32} =0.
\]
$\bullet$ (Estimates for $\mathcal I_{33}$): Note that weak-$*$ convergence of $(\Psi^j,\kappa^j)$ in $L^\infty(0,T; W^{1,q}(\bbsd))$ implies weak convergence in $L^2(0,T;L^2(\bbsd))$ to the same limit. Since $L^2(\bbsd)$ is the Hilbert space, the norm in  $L^2(0,T; L^2(\bbsd))$ is weakly lower semicontinuous. Hence, we have
\begin{align*}
\liminf_{j\to\infty} \mathcal I_{33} &=  \liminf_{j\to\infty} \left( \frac{\lambda_1}{2} ( \|\Psi^j\|_{L^2(0,T;L^2(\bbsd))}^2 - \|\Psi^*\|_{L^2(0,T; L^2(\bbsd))}^2) \right. \\
&\hspace{4cm} +  \left. \frac{\lambda_2}{2} ( \|\kappa^j\|_{L^2(0,T;L^2(\bbsd))}^2 - \|\kappa^*\|_{L^2(0,T; L^2(\bbsd))}^2)\right) \geq0.
\end{align*}
To this end, we finally obtain
\begin{align*}
\lim_{j\to\infty} \left(  \tilde{\mathcal J}(f^N[\Psi^j,\kappa^j], \Psi^j, \kappa^j) - \tilde{\mathcal J}(f^N[\Psi^*,\kappa^*], \Psi^*, \kappa^*) \right) \geq0.
\end{align*} 
This completes the proof.

\subsection{Proof of Theorem \ref{T1.1}(2)}

We proceed to the proof of Theorem \ref{T1.1}(2). \newline

\noindent\textit{Proof of Theorem \ref{T1.1}(2):}  
 It follows from the classical Csisz\'{a}r-Kullback-Pinsker inequality \cite[Chapter 22]{V08} that $L^1$ distance between two functions can be bounded by their relative entropy:
\begin{equation} \label{CKP}
\|f^{N;1}- f \|_{L^\infty(0,T; L^1(\bbsd))}^2 \leq 2\| \mathcal H (f^N | f^{\otimes N}(t)) \|_{L^\infty(0,T)}
\end{equation}
 where the relative entropy is defined as 
\[
\mathcal{H}(f^N \ | f^{\otimes N}) := \frac1N \int_{(\bbs^{d-1})^N} f^N(t, \omega_1, \dots, \omega_N) \log \frac{f^N(t, \omega_1, \dots, \omega_N)}{f^{\otimes N}(t, \omega_1, \dots, \omega_N)}\,d\omega_1 \dots d\omega_N.
\]
One can easily verify that  \eqref{CKP} still holds when the underlying space is given as $\bbsd$. 
Here, we introduce $d\omega_{1:N} := d\omega_1 \dots d\omega_N$ for simplicity. Then, the direct estimates yield
\begin{align*}
\frac{d}{dt}&\mathcal{H}(f^N \ | f^{\otimes N}) \\
&= \frac1N \int_{(\bbsd)^N} \pa_t f^N \log \frac{f^N}{f^{\otimes N}}\,d\omega_{1:N} - \frac1N  \int_{(\bbsd)^N} f^N \frac{\pa_t f^{\otimes N}}{f^{\otimes N}}\,d\omega_{1:N}  \\
&= \frac1N \sum_{i=1}^N \int_{(\bbsd)^N}   \lt(\Psi(t,\omega_i) \omega_i + \frac{\kappa(t,\omega_i)}{N}\sum_{k=1}^N \mathbb{P}(\omega_i) \omega_k \rt) f^N \cdot \nabla_{\omega_i} \log \frac{f^N}{f^{\otimes N}} \,d\omega_{1:N}\\
&\quad - \frac{\sigma}{N}\sum_{i=1}^N  \int_{(\bbsd)^N}  \nabla_{\omega_i} f^N \cdot \nabla_{\omega_i} \log \frac{f^N}{f^{\otimes N}} \,d\omega_{1:N} \\
&\quad - \frac1N \sum_{i=1}^N \int_{(\bbsd)^N}  \lt( \Psi(t,\omega_i)\omega_i  + \kappa(t,\omega_i) L[f](\omega_i) \rt)f^{\otimes N} \cdot \nabla_{\omega_i} \lt(\frac{f^N}{f^{\otimes N}} \rt) \,d\omega_{1:N}\\
&\quad + \frac{\sigma}{N}\sum_{i=1}^N \int_{(\bbsd)^N}  \nabla_{\omega_i} f^{\otimes N} \cdot \nabla_{\omega_i} \lt(\frac{f^N}{f^{\otimes N}} \rt) \,d\omega_{1:N}\\
&= \frac1N \sum_{i=1}^N \int_{(\bbsd)^N}  \kappa(t,\omega_i)\lt(\frac1N \sum_{k=1}^N \bbp(\omega_i)\omega_k - L[f](\omega_i) \rt) f^N \cdot \nabla_{\omega_i} \log \frac{f^N}{f^{\otimes N}} \,d\omega_{1:N}\\
&\quad  -\frac{\sigma}{N} \sum_{i=1}^N  \int_{(\bbsd)^N}  f^N (\nabla_{\omega_i} \log f^N)\cdot \nabla_{\omega_i} \log \frac{f^N}{f^{\otimes N}}\,d\omega_{1:N} \\
&\quad+ \frac{\sigma}{N}\sum_{i=1}^N \int _{(\bbsd)^N}f^N \nabla_{\omega_i} \log f^{\otimes N} \cdot \nabla_{\omega_i} \log \frac{f^N}{f^{\otimes N}}\,d\omega_{1:N}\\
&= \frac1N \sum_{i=1}^N  \int_{(\bbsd)^N}  \kappa(t,\omega_i)\lt(\frac1N \sum_{k=1}^N \bbp(\omega_i)\omega_k - L[f](\omega_i) \rt) f^N \cdot \nabla_{\omega_i} \log \frac{f^N}{f^{\otimes N}}\,d\omega_{1:N}\\
&\quad - \frac{\sigma}{N}\sum_{i=1}^N  \int_{(\bbsd)^N}  f^N \lt|\nabla_{\omega_i}\log \frac{f^N}{f^{\otimes N}} \rt|^2\,d\omega_{1:N}\\
&\le \frac{C}{N}\sum_{i=1}^N  \int_{(\bbsd)^N}  \lt|\frac1N \sum_{k=1}^N \bbp(\omega_i)\omega_k - L[f](\omega_i) \rt|^2 f^N\,d\omega_{1:N}\\
&\quad  - \frac\sigma{2N}\sum_{i=1}^N  \int_{(\bbsd)^N}  f^N \lt|\nabla_{\omega_i}\log \frac{f^N}{f^{\otimes N}} \rt|^2\,d\omega_{1:N},
\end{align*}
where we used Young's inequality. Here, we may write

\[
\frac{C}{N}\sum_{i=1}^N \int \lt|\frac1N \sum_{k=1}^N \bbp(\omega_i)\omega_k - L[f](\omega_i) \rt|^2 f^N = \frac{C}{N}\sum_{i=1}^N \bbe \lt[\lt|\frac1N \sum_{k=1}^N \mathbb{P}(x_i)x_k - L[f](x_i)\rt|^2 \rt],
\]
where the expectation is taken over the law $f^N$ (and also over the law $f^{\otimes N}$). Thus, we can get

\[\begin{aligned}
\frac{1}{N}\sum_{i=1}^N \bbe \lt[\lt|\frac1N \sum_{k=1}^N \mathbb{P}(x_i)x_k - L[f](x_i)\rt|^2 \rt] &\le \frac 3N \sum_{i=1}^N \bbe\lt[\lt| \frac1N \sum_{k=1}^N \bbp(x_i)x_k - \frac1N \sum_{k=1}^N \bbp(\bar{x}_i) \bar{x}_k\rt|^2 \rt]\\
&\quad + \frac3N \sum_{i=1}^N \bbe\lt[ \lt|\frac1N \sum_{k=1}^N \bbp(\bar{x}_i) \bar{x}_k - L[f](\bar{x}_i) \rt|^2 \rt]\\
&\quad + \frac3N \sum_{i=1}^N \bbe\lt[\lt|L[f](\bar{x}_i) - L[f](x_i)\rt|^2 \rt]\\
&=: \mathcal{I}_{41} + \mathcal{I}_{42} + \mathcal{I}_{43}.
\end{aligned}\] 
Below, we estimate $\mathcal I_{4k}$ for $k=1,2,3$, respectively. \newline

\noindent $\bullet$ (Estimates of $\mathcal I_{41}$ and $\mathcal I_{43}$): For $\mathcal{I}_{41}$, we   use the previous estimate, the facts $|\bbp(x)|\le 1$ on $\bbs^{d-1}$  and $|x_i| = |\bar{x}_i| =1$ to yield

\[\begin{aligned}
\mathcal{I}_{41} &\le \frac{6}{N}\sum_{i=1}^N \bbe\lt[\lt| \frac1N \sum_{k=1}^N \bbp(x_i)(x_k-\bar{x}_k) \rt|^2 \rt]  + \frac{6}{N}\sum_{i=1}^N\bbe\lt[\lt| \frac1N \sum_{k=1}^N (\bbp(x_i) - \bbp({\bar x_k}))\bar{x}_k \rt|^2 \rt]\\
&\le \frac CN.
\end{aligned}\]
By a similar argument, $\mathcal I_{43}$ can be estimated as
\[\begin{aligned}
\mathcal{I}_{43} & \le \frac3N \sum_{i=1}^N \bbe\lt[\lt|\int_{\bbs^{d-1}} \lt(\langle x_i , \omega_*\rangle x_i - \langle \bar{x}_i, \omega_* \rangle \bar{x}_i \rt) f(t,\omega_*) d\omega_*\rt|^2 \rt] \le \frac CN,
\end{aligned}\]
where $C$ depends on the $L^1$ norm of $f$. \newline

\noindent $\bullet$ (Estimates for $\mathcal I_{42}$): For $\mathcal{I}_{42}$, we have

\[\begin{aligned}
\mathcal{I}_{42} &= \frac3N \sum_{i=1}^N \bbe \lt[\lt|\bbp(\bar{x}_i) \lt( \frac1N \sum_{k=1}^N \bar{x}_k - \int_{\bbs^{d-1}}\omega_* f\,d\omega_*\rt) \rt|^2 \rt]\\
&\le C \bbe \lt[\lt|\lt( \frac1N \sum_{k=1}^N \bar{x}_k - \int_{\bbs^{d-1}}\omega_* f\,d\omega_*\rt) \rt|^2 \rt].
\end{aligned}\]
The right-hand side is nothing but the variation of the sample mean and hence,
\[
\mathcal{I}_{42} \le C \frac{\mbox{var}(\bar{x})}{N} \le \frac CN.
\]
Thus, we arrive at
\[\begin{aligned}
\frac{d}{dt}\mathcal{H}(f^N | f^{\otimes N}) + \frac{\sigma}{2N}\sum_{i=1}^N \int_{(\bbsd)^N} f^N \lt| \nabla_{\omega_i}\log \frac{f^N}{f^{\otimes N}}\rt|^2 \,d\omega_{1:N} \le \frac CN.
\end{aligned}\]
Finally, direct integration gives the desired result. \qed

\subsection{Proof of Theorem \ref{T1.1}(3)}
We begin with two elementary lemmas. 

\begin{lemma} \label{L4.1} 
Assume that the initial data is a probability density $f_0 \in H^1(\bbsd)$. For a sequence $(\Psi^N, \kappa^N)_{N\in \bbn} \in \mathcal U^2$, there exists a weakly convergent subsequence $(\Psi^{N_k}, \kappa^{N_k}) \rightharpoonup (\bar \Psi,\bar \kappa) \in \mathcal U^2$ such that the corresponding solution $f[\Psi^{N_k},\kappa^{N_k}]$ converges to $f[\bar \Psi,\bar \kappa]$  in the sense that
\[
f[\Psi^{N_k},\kappa^{N_k}] \to f[\bar \Psi,\bar \kappa] \quad \textup{in $C(0,T;L^2(\bbsd))$}.
\]
In addition, there exists a uniform constant $C>0$ not depending on $N\in \bbn$ such that
\[
\|f[\Psi^N, \kappa^N] \|_{C(0,T; H^1(\bbsd))} + \|f [\bar \Psi,\bar \kappa]\|_{C(0,T; H^1(\bbsd))} \leq C,\quad N\in \bbn.
\]
\end{lemma}

\begin{proof}
In fact, we notice that all estimates in Theorem \ref{T1.1}(1) do not depend on $j$. Hence, straightforward adaptation of Theorem \ref{T1.1}(1) gives the desired estimates. 
\end{proof}

\begin{lemma} \label{L4.2} 
Assume that the initial data is the probability density $f_0\in H^1(\bbsd)$ and that $(\hat \Psi^N, \hat \kappa^N)_{N\in \bbn}$  is a sequence in $\mathcal U^2$. Then, we have
\[
\lim_{N\to\infty} \left( \tilde{\mathcal J}(f^N[\hat \Psi^N, \hat \kappa^N], \hat \Psi^N, \hat \kappa^N)  -  \mathcal J(f[\hat \Psi^N, \hat \kappa^N], \hat \Psi^N, \hat \kappa^N)  \right) =0.
\]
\end{lemma}

\begin{proof}
By the definitions of $\tilde{\mathcal J}$ and $\mathcal J$, 
\begin{align*}
&\left| \tilde{\mathcal J}(f^N[\hat \Psi^N, \hat \kappa^N], \hat \Psi^N, \hat \kappa^N)  -  \mathcal J(f[\hat \Psi^N, \hat \kappa^N], \hat \Psi^N, \hat \kappa^N)  \right| \\
&\leq  \frac{\alpha}{2} \int_0^t \int_\bbsd \Big| (f^{N;1}[\hat \Psi^N, \hat \kappa^N](t,\omega) - z(t,\omega))^2 - ( f[\hat \Psi^N,\hat \kappa^N](t,\omega) - z(t,\omega))^2 \Big| d\omega dt \\
&\quad + \frac{\beta}{2} \int_\bbsd \Big| (f^{N;1}[\hat \Psi^N, \hat \kappa^N](T,\omega) - z(T,\omega))^2 - ( f[\hat \Psi^N,\hat \kappa^N](T,\omega) - z(T,\omega))^2 \Big| d\omega \\
& = :\mathcal I_{51} + \mathcal I_{52}.
\end{align*}
Since we have
\begin{align*}
\| f^{N;1}[\hat \Psi^N,\hat \kappa^N]\|_{L^\infty(0,T; H^1(\bbsd))} \leq C,\quad\|f[\hat \Psi^N,\hat \kappa^N]\|_{L^\infty(0,T; H^1(\bbsd))} \leq C,
\end{align*}
we observe that
\begin{align*}
  \mathcal J_{51}   & \leq  \frac{\alpha}{2}  \| f^{N;1}[\hat \Psi^N,\hat \kappa^N] - f [\hat \Psi^N,\hat \kappa^N]\|_{L^\infty(0,T; L^1(\bbsd))} \\
&\quad \times   \|f^{N;1}[\hat \Psi^N,\hat \kappa^N] + f [\hat \Psi^N,\hat \kappa^N] -2z\|_{L^1(0,T; L^\infty(\bbsd))}  \\
&\leq \frac{C\alpha}{2}  \| f^{N;1}[\hat \Psi^N,\hat \kappa^N] - f [\hat \Psi^N,\hat \kappa^N]\|_{L^\infty(0,T; L^1(\bbsd))} .
\end{align*}
Hence, we have
\[
\lim_{N\to\infty} \mathcal J_{51}=0.
\]
By a similar argument, we also have
\[
\lim_{N\to\infty} \mathcal J_{52}=0.
\]
\end{proof}

Finally, the last assertion of Theorem \ref{T1.1} is provided. \newline

\noindent \textit{Proof of Theorem \ref{T1.1}(3):}  
 It follows from Theorem \ref{T1.1}(1) that for any fixed $N$, we know that there exist minimizers of $\tilde{\mathcal J}(f^N[\Psi, \kappa], \Psi, \kappa)$, denoted by $(\Psi^N, \kappa^N)$, i.e., 
\[
\tilde{\mathcal J}(f^N[\Psi^N,\kappa^N],\Psi^N,\kappa^N) = \min_{(\Psi,\kappa) \in\mathcal U^2} \tilde{\mathcal J}(f^N[\Psi, \kappa], \Psi, \kappa) .
\]
Thanks to Lemma \ref{L4.1}, there exists a weak limit $(\bar \Psi,\bar \kappa)$ of a subsequence $(\Psi^{N_k},\kappa^{N_k})$ of $(\Psi^N,\kappa^N)$ such that
\begin{equation} \label{E-10}
f[\Psi^{N_k},\kappa^{N_k}] \to f[\bar \Psi,\bar \kappa] \quad \textup{in $C(0,T;L^2(\bbsd))$}. 
\end{equation} 
In order to achieve the desired result, we need to show that
\begin{equation} \label{E-11}
\liminf_{k\to\infty} \tilde{\mathcal J} (f^{N_k} [\Psi^{N_k}, \kappa^{N_k}], \Psi^{N_k}, \kappa^{N_k}) \geq  \mathcal J(f[\bar\Psi, \bar\kappa], \bar\Psi, \bar\kappa).
\end{equation}
It suffices to show that
\begin{align*}
&\liminf_{k\to\infty}  \left(  \tilde{\mathcal J} (f^{N_k} [\Psi^{N_k}, \kappa^{N_k}], \Psi^{N_k}, \kappa^{N_k}) - {\mathcal J} (f [\Psi^{N_k}, \kappa^{N_k}], \Psi^{N_k}, \kappa^{N_k}) \right) \\
& \quad + \liminf_{k\to\infty}  {\mathcal J} (f   [\Psi^{N_k}, \kappa^{N_k}], \Psi^{N_k}, \kappa^{N_k}) -  {\mathcal J}([\bar\Psi, \bar\kappa], \bar\Psi, \bar\kappa) \geq 0.
\end{align*}
The first term vanishes due to Lemma \ref{L4.2}. For the second term, 
\begin{align*}
& {\mathcal J} (f   [\Psi^{N_k}, \kappa^{N_k}], \Psi^{N_k}, \kappa^{N_k}) -  \mathcal J(f[\bar\Psi, \bar\kappa], \bar\Psi, \bar\kappa) \\
& = \frac{\alpha}{2} \int_0^T \int_\bbsd \left(  f[\Psi^{N_k},\kappa^{N_k}](t,\omega) - f(\bar \Psi,\bar \kappa)(t,\omega)       \right) \left(  f[\Psi^{N_k},\kappa^{N_k}](t,\omega) +  f(\bar \Psi,\bar \kappa)(t,\omega) -2z(t,\omega)        \right)       d\omega dt  \\
& \quad +\frac{\beta}{2}\int_\bbsd   \left(  f[\Psi^{N_k},\kappa^{N_k}](T,\omega) - f(\bar \Psi,\bar \kappa)(T,\omega)       \right) \left(  f[\Psi^{N_k},\kappa^{N_k}](T,\omega) +  f(\bar \Psi,\bar \kappa)(T,\omega) -2z(T,\omega)        \right)       d\omega                      \\
&\quad + \frac12 \int_0^T \int_\bbsd  \left[ \lambda_1 \Big( (( \Psi^{N_k}(t,\omega))^2 - (\bar \Psi(t,\omega))^2  \Big) +  \lambda_2 \Big( (( \kappa^{N_k}(t,\omega))^2 - (\bar \kappa(t,\omega))^2)  \Big) \right] d\omega dt  \\
& = : \mathcal I_{61} + \mathcal I_{62} +\mathcal I_{63}. 
\end{align*}
For the estimates of  $\mathcal I_{61}$ and $\mathcal I_{62}$,  we observe
\begin{align*}
\mathcal I_{61} & \leq \frac{\alpha}{2} \|  f[\Psi^{N_k},\kappa^{N_k}] - f[\bar \Psi,\bar \kappa]  \|_{L^2(0,T; L^2(\bbsd))} \\
&\quad \times \left(         \|  f[\Psi^{N_k},\kappa^{N_k}] - z  \|_{L^2(0,T; L^2(\bbsd))} +  \|  z- f[\bar \Psi,\bar \kappa]  \|_{L^2(0,T; L^2(\bbsd))}                 \right) 
\end{align*}
and
\begin{align*}
\mathcal I_{62} & \leq \frac{\alpha}{2} \|  f[\Psi^{N_k},\kappa^{N_k}](T) - f[\bar \Psi,\bar \kappa](T)  \|_{L^2(\bbsd)} \\
&\quad \times \left(         \|  f[\Psi^{N_k},\kappa^{N_k}] (T) - z (T)  \|_{ L^2(\bbsd)} +  \|  z(T)- f[\bar \Psi,\bar \kappa](T)  \|_{ L^2(\bbsd)}                 \right).
\end{align*}
Then, we use \eqref{E-10} and \eqref{L4.1} to find
\[
\lim_{k\to\infty} \mathcal I_{61} =0 \quad \textup{and}\quad \lim_{k\to\infty} \mathcal I_{62} =0.
\]
For $\mathcal I_{63}$, similar to the Estimates for $\mathcal I_{33}$, we observe
\begin{align*}
\liminf_{k\to\infty} \mathcal I_{63} & = \liminf_{k\to\infty} \left(     \frac{\lambda_1}{2}  \Big(  \|\Psi^{N_k}\|_{L^2(0,T; L^2(\bbsd))}^2 - \|\bar \Psi\|_{L^2(0,T;L^2(\bbsd))}^2      \Big)    \right. \\
&\quad \left. +\frac{\lambda_2}{2}  \Big(  \|\kappa^{N_k}\|_{L^2(0,T; L^2(\bbsd))}^2 - \|\bar \kappa\|_{L^2(0,T;L^2(\bbsd))}^2      \Big)                    \right)\geq0
\end{align*}
where we used the norm in $L^2(0,T;L^2(\bbsd))$ is weakly lower semicontinuous. This establishes \eqref{E-11}. Finally, for any $(\hat \Psi,\hat \kappa)\in \mathcal U^2$, it follows from similar arguments to Lemma \ref{L4.2} that 
\begin{align*}
\mathcal J ( f[\hat \Psi,\hat \kappa], \hat \Psi,\hat \kappa) &= \lim_{N\to\infty} \tilde{\mathcal J} ( f^N[\hat \Psi,\hat \kappa], \hat \Psi,\hat \kappa) = \lim_{k\to\infty} \tilde{\mathcal J} ( f^{N_k}[\hat \Psi,\hat \kappa], \hat \Psi,\hat \kappa)  \\
&\geq \liminf_{k\to\infty} \min_{(\Psi,\kappa)\in \mathcal U^2} \tilde{\mathcal J} ( f^{N_k}[  \Psi,  \kappa],   \Psi,  \kappa) \\
&\geq \mathcal J( f[\bar \Psi,\bar \kappa], \bar \Psi,\bar \kappa).
\end{align*}
This shows
\[
\mathcal J( f[\bar \Psi,\bar \kappa], \bar \Psi,\bar \kappa) = \min_{(\Psi,\kappa)\in \mathcal U^2} J( f[  \Psi,  \kappa],   \Psi,  \kappa)
\]
and this completes the proof of Theorem \ref{T1.1}.

\end{document}